\documentclass[11pt,oneside]{amsart}
\usepackage{amsthm, amsmath, amssymb, amsfonts, url, booktabs, tikz, setspace, fancyhdr, enumerate}
\usepackage{enumitem}
\usepackage[margin = 1in]{geometry}
\usepackage[czech,english]{babel} 
\usepackage[hidelinks]{hyperref}
\usepackage[capitalise]{cleveref}
\usepackage{tikz}
\usepackage{comment}
\usepackage{xcolor}
\usepackage[noadjust]{cite}

\usepackage[textsize=scriptsize,colorinlistoftodos]{todonotes}

\newtheorem{theorem}{Theorem}[section]
\newtheorem{proposition}[theorem]{Proposition}
\newtheorem{lemma}[theorem]{Lemma}
\newtheorem{fact}[theorem]{Fact}
\newtheorem{corollary}[theorem]{Corollary}
\newtheorem{conjecture}[theorem]{Conjecture}
\newtheorem{claim}[theorem]{Claim}
\newtheorem{question}[theorem]{Question}
\newtheorem{problem}[theorem]{Problem}
\newtheorem{algo}[theorem]{Algorithm}
\newtheorem{thm}[theorem]{Theorem}
\newtheorem{lem}[theorem]{Lemma}

\theoremstyle{definition}
\newtheorem{definition}[theorem]{Definition}
\newtheorem{example}[theorem]{Example}

\newtheorem{constr}[theorem]{Construction}

\theoremstyle{remark}
\newtheorem*{remark}{Remark}

\newcommand{\boldsubsection}[1]{\par\medskip\noindent\textbf{#1.}}

\let\oldln\ln

\renewcommand{\ln}{\log}

\renewcommand{\geq}{\geqslant}
\renewcommand{\ge}{\geqslant}
\renewcommand{\leq}{\leqslant}
\renewcommand{\le}{\leqslant}

      \newcommand{\g}{\mathcal{F}}
      \newcommand{\f}{\mathcal{G}}
      \newcommand{\h}{\mathcal{H}}

      \newcommand{\mf}{\mathcal{M}}
      \newcommand{\nf}{\mathcal{N}}

      \newcommand{\m}{M}
      \newcommand{\n}{N}
      \newcommand{\set}{S}

      \newcommand{\wb}{\mathfrak{B}_c}
      \newcommand{\ww}{\omega}
      \newcommand{\win}{\omega_{in}}
      \newcommand{\wout}{\omega_{out}}

    \newcommand{\md}{\delta}

\newcommand{\G}{J}

\newcommand{\U}{u}
\newcommand{\w}{\omega }
\newcommand{\ff}{f}
\newcommand{\miniweight}{mini-weight}
\newcommand{\Q}{\mathcal{Q}}
\newcommand{\q}{Q}
\newcommand{\e}{\varepsilon}
\newcommand{\C}{c}
\newcommand{\F}{\mathcal{F}}
\newcommand{\PHI}{\varphi}
\newcommand{\SIG}{\sigma}

\begin{document}
\title{Exact results on traces of sets}

\author[Mingze Li]{Mingze Li}
\address{School of Mathematical Sciences, University of Sciences and Technology of China, Hefei,
China}
\email{lmz10@mail.ustc.edu.cn}

\author[Jie Ma]{Jie Ma}
\address{School of Mathematical Sciences, University of Sciences and Technology of China, Hefei,
China}
\email{jiema@ustc.edu.cn}

\author[Mingyuan Rong]{Mingyuan Rong}
\address{School of Mathematical Sciences, University of Sciences and Technology of China, Hefei,
China}
\email{rong\_ming\_yuan@mail.ustc.edu.cn}

\date{}

\begin{abstract}
For non-negative integers $n$, $m$, $a$ and $b$, we write $\left( n,m \right) \rightarrow \left( a,b \right)$ if for every family $\mathcal{F}\subseteq 2^{[n]}$ with $|\mathcal{F}|\geqslant m$ there is an $a$-element set $T\subseteq [n]$ such that $\left| \mathcal{F}_{\mid T} \right| \geqslant b$, where $\mathcal{F}_{\mid T}=\{ F \cap T : F \in \mathcal{F}  \}$.
A longstanding problem in extremal set theory asks to determine $m(s)=\lim_{n\rightarrow +\infty}\frac{m(n,s)}{n}$,
where $m(n,s)$ denotes the maximum integer $m$ such that $\left( n,m \right) \rightarrow \left( n-1,m-s \right)$ holds for non-negatives $n$ and $s$.
In this paper, we establish the exact value of $m(2^{d-1}-c)$ for all $1\leqslant c\leqslant d$ whenever $d\geqslant 50$, thereby solving an open problem posed by Piga and Sch\"{u}lke.
To be precise, we show that 
\[
m(n,2^{d-1}-c)=\begin{cases}
         \frac{2^{d}-c}{d}n &\mbox{ for } 1\leq c\leq d-1 \mbox{ and } d\mid n \\
         \frac{2^{d}-d-0.5}{d}n &\mbox{ for } c=d \mbox{ and } 2d\mid n 
         \end{cases} 
\]
holds for $d\geq 50$.
Furthermore, we provide a proof that confirms a conjecture of Frankl and Watanabe from 1994, demonstrating that $m(11)=5.3$.
\end{abstract}

\maketitle

\section{Introduction}
Let $V$ be a set with $n$ elements and let $\mathcal{F}$ be a family of subsets of $V$. 
For a subset $T$ of $V$, the $trace$ of $\mathcal{F}$ on $T$ is defined by $\mathcal{F}_{|T}=\{F\cap T:F\in \mathcal{F}\}$.
For positive integers $n$, $m$, $a$, and $b$, we write
$$
(n, m) \rightarrow(a, b)
$$
and we say $(n, m)$ \emph{arrows} $(a, b)$, if for every family $\mathcal{F} \subseteq 2^{V}$ with $|\mathcal{F}| \geqslant m$ and $|V|=n$ there is an $a$-element set $T \subseteq V$ such that $\left|\mathcal{F}_{\mid T}\right| \geqslant b$.

\subsection{Background}
Early investigations into this arrowing relation focus on determining the minimum number $M(n)$ such that for all $m > M(n)$, $(n, m) \rightarrow (a, b)$ holds for fixed integers $a$ and $b$.
Answering a conjecture by Erd\H{o}s \cite{FP1991}, Sauer \cite{S1972}, Shelah and Perles \cite{SP1972}, and Vapnik and Červonenkis \cite{V1968,VC2015} independently proved the now-called Sauer-Shelah lemma, which states that $(n,m)\rightarrow (s, 2^s)$ holds whenever $m>\sum_{i=0}^{s-1} \binom{n}{i}$. 
Bondy and Lov\'asz conjectured and Frankl \cite{F1983} demonstrated that $(n, m) \rightarrow(3,7)$ holds when $m>\lfloor\frac{n^2}{4}\rfloor+n+1$.
The most recent result in this direction was obtained by Frankl and Wang \cite{FW2024}, who established an optimal bound on $m$ for which $(n, m) \rightarrow(4,13)$ holds.
Further information can be found in the book by Gerbner and Patkós \cite{GP2018}.

Another intriguing problem that has captured considerable attention in relation to the arrowing relation is the following one posed by Füredi and Pach \cite{FP1991}.
\begin{problem}[Füredi and Pach \cite{FP1991}]\label{1.1}
For a fixed non-negative integer $s$, find the largest $m(s)$ such that $(n, m) \rightarrow(n-1, m-s)$ holds for all $m \leq(m(s)-o(1)) n$.
\end{problem}
An exact version of this problem can be found in the book by Frankl and Tokushige \cite{FT2018}.
\begin{problem}[Frankl and Tokushige \cite{FT2018}, Problem 3.8]\label{1.2}
For any non-negative integers $n$ and $s$, determine or estimate the maximum value $m=m(n, s)$ such that $(n, m) \rightarrow(n-1, m-s).$
\end{problem}
It can be verified that $m(s)$ is well-defined, and in fact it satisfies $m(s)=\lim_{n\rightarrow +\infty}\frac{m(n,s)}{n}$ (see e.g. \cite{WF1994}).
Hence, in order to determine $m(s)$, it is sufficient to determine $m(n,s)$ for an infinite sequence of integers $n$.

The investigation of Problem~\ref{1.1} initially focused on determining $m(s)$ for small values of $s$. Bondy \cite{B1972} established that $m(0)=1$, while Bollob\'as \cite{L2007} determined $m(1)=\frac{3}{2}$.
Further progress was made by Frankl \cite{F1983}, who found that $m(2)=2$ and $m(3)=\frac{7}{3}$. Watanabe \cite{W1991} extended these results by determining $m(4)=\frac{17}{6}$, $m(5)=\frac{13}{4}$, and $m(6)=\frac{7}{2}$.
Continuing the exploration, Frankl and Watanabe \cite{WF1994} proved $m(9)=\frac{65}{14}$, while Watanabe \cite{W1995} subsequently determined $m(10)=5$ and $m(13)=\frac{29}{5}$.
A very recent result of Piga and Sch\"{u}lke \cite{PS2021} implies that $m(12)=\frac{28}{5}$.

Based on the developments outlined above, 
the values of $m(s)$ are known for all $s\leq 16$, except for $s = 11$. 
In 1994, Frankl and Watanabe \cite{WF1994} proposed the following conjecture (expressed in a different but equivalent notation).

\begin{conjecture}[Frankl and Watanabe \cite{WF1994}, Conjecture~3]\label{conj_for_m11}
$m(11)=5.3.$
\end{conjecture}

After the initial investigations, subsequent explorations of this problem shift towards determining $m(n,s)$ for values of $s$ given by $s=2^{d-1}-c$ with $c\geq 0$.
The results of these explorations can be summarized in the following theorem.

\begin{theorem}\label{1.3}
 Let $d, n \in \mathbb{N}$. Then the following hold:
\begin{itemize}
     \item[(i)] For $d\geq2$, $d \mid n$ and $2\mid n$, $m(n, 2^{d-1}-0)=\frac{2^{d}-1}{d}n+\frac{n}{2}$.
     \item[(ii)] For $d\geq2$ and $d \mid n$, $m(n, 2^{d-1}-1)=\frac{2^{d}-1}{d}n$.
     \item[(iii)] For $d\geq3$ and $d \mid n$, $m(n, 2^{d-1}-2)=\frac{2^{d}-2}{d}n$.
     \item[(iv)] For $d\geq4$ and $d \mid n$, $m(n, 2^{d-1}-3)=\frac{2^{d}-3}{d}n$.
     \item[(v)] For $d\geq5$ and $d \mid n$, $m(n, 2^{d-1}-4)=\frac{2^{d}-4}{d}n$.
\end{itemize}
\end{theorem}
To provide proper attribution, it should be noted that Frankl \cite{F1983} proved (ii), while Frankl and Watanabe \cite{WF1994} proved (i) and (iii). 
Recently, Piga and Sch\"{u}lke \cite{PS2021} demonstrated (iv) and (v). 
Additionally, Piga and Sch\"{u}lke established a more general result, stated as follows.   

\begin{theorem}[Piga and Sch\"{u}lke \cite{PS2021}]\label{Pigatheorem}
Let $d, c, n \in \mathbb{N}$ with $c\leq \frac{d}{4}$. If $d \mid n$, then
$$
m(n, 2^{d-1}-c)=\frac{2^{d}-c}{d}n, \mbox{ which implies that } m(2^{d-1}-c)=\frac{2^{d}-c}{d}.
$$
\end{theorem}

Piga and Sch\"{u}lke \cite{PS2021} raised an intriguing problem regarding the optimal upper bound for $c$ in Theorem \ref{Pigatheorem}. 

\begin{problem}[Piga and Sch\"{u}lke \cite{PS2021}]\label{problem_for_c0}
For fixed $d$, find the maximum integer $c_0(d)$ such that for every $1\leq c \leq c_0(d)$, it holds that $$m(2^{d-1}-c)=\frac{2^{d}-c}{d}.$$
\end{problem}

In fact, they also provided a construction (see Construction~\ref{construction_c=d}) which shows $m(2^{d-1}-d)\leq \frac{2^d-d-\frac{1}{2}}{d}$. 
Consequently, it implies that the maximum $c_0(d)$ in Problem~\ref{problem_for_c0} satisfies $c_0(d)\leq d-1$.

\subsection{Our results}
In this paper, we address the aforementioned problems by introducing some novel concepts and utilizing advanced techniques.
We begin by presenting our first result, which confirms Conjecture \ref{conj_for_m11} posed by Frankl and Watanabe \cite{WF1994}.

\begin{theorem}\label{zhudingli_d=5}
    For $n \in \mathbb{N}$ with $10\mid n$, $$m(n, 11)=m(n, 2^{5-1}-5)=\frac{2^5-5-\frac{1}{2}}{5}n=5.3n.$$
Therefore $m(11)=5.3$.
\end{theorem}

Our second result determines the exact values of $m(2^{d-1}-c)$ for all $1\leq c\leq d$ when $d\geq 50$. 
Moreover, it extends the statement of Theorem~\ref{Pigatheorem} to the optimal range of $c$.

\begin{theorem}\label{zhudingli}
Let $d, n \in \mathbb{N}$ with $d\geq 50$. Then the following hold: 
\begin{itemize}
\item For $2d \mid n$, we have 
$$m(n, 2^{d-1}-d)=\frac{2^d-d-\frac{1}{2}}{d}n, \mbox{ which implies that } m(2^{d-1}-d)=\frac{2^{d}-d-\frac12}{d}.$$
\item For all $1\leq c\leq d-1$ and $d \mid n$, we have
$$m(n, 2^{d-1}-c)=\frac{2^{d}-c}{d}n, \mbox{ which implies that } m(2^{d-1}-c)=\frac{2^{d}-c}{d}.$$
    \end{itemize}
\end{theorem}

We point out that in the first case that $c=d$, essentially we can characterize all extremal families (see the statement of Theorem~\ref{thm000} and the explanations before Subsection~\ref{subsec:iso-pile}).  
As a corollary, Theorem~\ref{zhudingli} provides a precise answer to Problem \ref{problem_for_c0} for $d\geq 50$.

\begin{corollary}
For $d\geq 50$, it holds that $c_0(d)=d-1$.
\end{corollary}

\boldsubsection{Structure of the paper}
This paper is organized as follows.
Section 2 provides the necessary preliminaries, including the introduction of notations, concepts, constructions, and several useful lemmas.
In Section 3, we provide a detailed proof of Theorem \ref{zhudingli_d=5}.
Section 4 presents a reduction that demonstrates how to derive Theorem \ref{zhudingli} from Theorem \ref{auxiliary_theorem}.
Section 5 is dedicated to the proof of Theorem~\ref{auxiliary_theorem}.
Lastly, in Section 6, we discuss some open problems.

\boldsubsection{Acknowledgement}
J. M. was supported in part by National Key Research and Development Program of China 2023YFA1010201 and National Natural Science Foundation of China grant 12125106.

\section{Preliminaries}

\subsection{Notation and hereditary family}

In this paper, we adopt the standard notation. 
We define $[n]$ as the set $\{1, 2, \ldots, n\}$, and $[n]_0$ as the set $\{0, 1, 2, \ldots, n\}$. 
Additionally, we denote $\mathbb{N}$ as the set of non-negative integers and $\mathbb{N}^{+}$ as the set of positive integers.
Unless otherwise specified, all logarithms are base $2$. 

The problems we discuss here can be formulated using both the languages of set theory and hypergraph theory.
Therefore, we will treat the family $\mathcal{F} \subseteq 2^{V}$ and the hypergraph $(V, \mathcal{F})$ as interchangeable without distinction.
Let $V$ be a finite set. For a family $\F \subseteq 2^{V}$ and $x\in V$, we define the \emph{link} $\F(x)$ of $x$ to be $\F(x)=\{F\setminus \{x\}:x\in F\in \F\}$. 
We define the \emph{degree} of $x$ in $\F$ to be $d_{\F}(x)=|\F(x)|$ and the \emph{minimum degree} of $\F$ to be $\delta(\F)=\min \{d_{\F}(x):x\in V\}$. 
The \emph{neighborhood} of $x$ means $N(x)=\bigcup_{x\in F\in \F}F$. 
Note that if $x$ is non-isolated (i.e., $d_\mathcal{F}(x)\geq 1$), then $x\in N(x)$. 
For $\mathcal{F^{\prime}}\subseteq \mathcal{F}$, $\mathcal{F}\setminus \mathcal{F^{\prime}}$ denotes the subfamily obtained from $\mathcal{F}$ by deleting every set $F\in \mathcal{F^{\prime}}$.

A family $\mathcal{F}$ is \emph{hereditary} if for every $F^{\prime} \subseteq F \in \mathcal{F}$, it holds that $F^{\prime} \in \mathcal{F}$. 
A classic result of Frankl \cite{F1983} implies that for fixed $a$ and $m$, the minimum of $\max_{T\in \binom{V}{a}} \left|\mathcal{F}_{\mid T}\right|$ over all families $\mathcal{F}\subseteq 2^V$ of size $m$ is achieved by hereditary families. 
Thus, the problems on arrow relation considered in this paper can be reduced to hereditary families. 
We can summarize this as the following lemma.

\begin{lemma}[\cite{F1983}]\label{2.1}
The following statements are equivalent for positive integers $n$, $m$, $a$ and $b$:
\begin{itemize}
    \item $(n,m)\to (a,b)$.
    \item For every $n$-set $V$ and every hereditary family $\mathcal{F} \subseteq 2^{V}$ with $|\mathcal{F}|=m$, there exists $T\subseteq V$ with $|T|=a$ such that $|\mathcal{F}_{|T}|\ge b$. 
\end{itemize}
\end{lemma}

A simple yet important observation is that for any hereditary family $\mathcal{F}\subseteq 2^{V}$ and $x\in V$, we have $|\mathcal{F}_{\mid V\setminus \{x\}}|=|\mathcal{F}|-d_{\mathcal{F}}(x)$. 
By combining this observation with Lemma~\ref{2.1}, we obtain a useful proposition that will be frequently utilized.

\begin{proposition}\label{dengjiatiaojian}
For positive integers $n$, $m$ and $s$, the following statements are all equivalent:
\begin{itemize}
\item $(n, m) \rightarrow(n-1, m-s)$.
\item $m\leqslant m(n,s)$. 
\item For any hereditary family $\mathcal{F}\subseteq 2^{[n]}$ with $|\mathcal{F}|\leqslant m$, there exists $x_{0}\in [n]$ such that $d_{\mathcal{F}}(x_0)\leqslant s$. 
\item For any hereditary family $\mathcal{F}\subseteq 2^{[n]}$ with $\delta(\mathcal{F})\geqslant s+1$, we have $|\mathcal{F}|\geqslant m+1$. 
\end{itemize}
\end{proposition}

\subsection{Constructions}
In this subsection, we provide constructions for two types of set families that serve as upper bounds for Theorems~\ref{zhudingli_d=5} and \ref{zhudingli}.
The first construction follows a natural approach.

\begin{constr}\label{construction_c<d}
Let $d\ge 5$, $1\leq c\leqslant d-1$ and $n=dk$, where $k$ is a positive integer. 
Let $V$ be a set of size $n$,
and let $U_{1},...,U_{k}$ form a partition of $V$ into sets of size $d$.
For every $i\in [k]$, arbitrarily pick a family $\mathcal{G}_{i}\subseteq 2^{U_{i}}$ with $|\mathcal{G}_{i}|=c-1$. 
Define
\[
\mathcal{F}(n,d,c) := \left\{ F\subseteq V : F\in 2^{U_{i}}\setminus \mathcal{G}_{i} \mbox{ for some $i\in [k]$}\right\}. 
\]
\end{constr}

For this construction, it is easy to check that $|\mathcal{F}(n,d,c)|=\frac{n}{d}(2^{d}-c)+1$ and $\delta(\mathcal{F}(n,d,c))\geqslant 2^{d-1}-c+1$. 
By Proposition~\ref{dengjiatiaojian},
this implies that
\[
m(n,2^{d-1}-c)\leqslant \frac{2^{d}-c}{d}n \mbox{ for any $d\geq 5, ~ d\mid n$ and $1\leq c\leqslant d-1$.} 
\]
In particular, we see that the upper bound of the second part of Theorem~\ref{zhudingli} holds. 

The second construction was given in \cite{PS2021},
which deals with the case when $c=d$.

\begin{constr}[\cite{PS2021}]\label{construction_c=d}
Let $d\ge 5$ and $n=2dk$, where $k$ is a positive integer. 
We define the family $\mathcal{F}_{0}(n,d)$ as follows. 
For $i\in [2k]$, we set $U_{i}=\{1+d(i-1),2+d(i-1),\dots,di\}$. Then $U_{1},...,U_{2k}$ provide a partition of $[n]$ into sets of size $d$. 
Define $\mathcal{F}_{0}(n,d):=\mathcal{G}\cup \mathcal{H}\cup \mathcal{I}$, where
\begin{align}
&\mathcal{G} = \left \{ S\subseteq V \mbox{: there exists some $i\in [2k]$ such that $S\subseteq U_{i}$ and $|S|\le d-2$} \right \} \notag \\
&\mathcal{H}=\left \{ U_{i}\setminus \left \{ di \right \} \mbox{: for $i\in [2k] $}  \right \} \notag \\
&\mathcal{I}=\left \{ \left \{ di,d(i+1) \right \} \mbox{: for $i\in \left \{ 1,3,5,...,2k-1 \right \} $}  \right \}.\notag 
\end{align}
\end{constr}

One can check that the number of edges of the family $\mathcal{F}_{0}(n,d)$ is given by
\begin{equation}
|\mathcal{F}_{0}(n,d)|=|\mathcal{G}|+|\mathcal{H}|+|\mathcal{I}|=\left(\frac{2^{d}-d-2}{d} n+1\right)+\frac{n}{d} +\frac{n}{2d} =\frac{2^{d}-d-\frac{1}{2} }{d} n+1.\notag
\end{equation}
Moreover, every vertex in $V$ has degree $2^{d-1}-d+1$. 
By Proposition~\ref{dengjiatiaojian},
this implies 
\[
m(n,2^{d-1}-d)\le \frac{2^{d}-d-\frac{1}{2} }{d}n 
\mbox{ for any $d\geq 5$ and $2d\mid n$. }
\]
Therefore, the upper bounds stated in Theorem~\ref{zhudingli_d=5} and the first part of Theorem~\ref{zhudingli} are satisfied.

\subsection{Colexicographic order}

For two finite sets $A,B\subseteq \mathbb{N}^{+}$, we say that $A\prec_{col}B$ or $A$ precedes $B$ in the {\it colexicographic order} if $\max(A\bigtriangleup B)\in B$. 
For $m\in \mathbb{N}$, we define $\mathcal{R}(m)$ to be the family containing the first $m$ finite subsets of $\mathbb{N}^{+}$ according to the colexicographic order. In particular, $\mathcal{R}(0)=\emptyset $. The following theorem due to Katona \cite{K1978} is a generalisation of the well-known Kruskal-Katona theorem. 

\begin{theorem}
[\cite{K1978}]
Let $f:\mathbb{N}\to \mathbb{R}$ be a monotone non-increasing function and let $\mathcal{F}$ be a hereditary family with $|\mathcal{F}|=m$. Then 
\begin{equation}
\sum_{F\in \mathcal{F}}f(|F|)\ge \sum_{R\in \mathcal{R}(m)}f(|R|).\notag
\end{equation}
\label{katona}
\end{theorem}

When applying Theorem \ref{katona} in our proofs, 
we often use $\F$ as the link of a vertex, and the function $f$ is commonly chosen as $f(k)=\frac{1}{k+1}$. 
For convenience, we define
\[
W(m):=\sum_{R\in \mathcal{R}(m)}\frac{1}{|R|+1}.
\]
Note that when $m=2^{d-1}$, we have the following expression
\[
W(2^{d-1})=\sum_{R\subseteq [d-1]}\frac{1}{|R|+1}=\sum_{i=0}^{d-1}\frac{\binom{d-1}{i}}{i+1}=\sum_{i=0}^{d-1}\frac{\binom{d}{i+1}}{d}=\frac{2^{d}-1}{d}. 
\]
The following lemma can be found in \cite{PS2021}.  
For the sake of completeness, we give a proof.
\begin{lemma}[\cite{PS2021}]
Let $d$ and $c$ be positive integers with $c\le 2^{d-2}$. Then we have 
\begin{equation}
W(2^{d-1}-c)\ge \frac{2^{d}-1}{d}-\frac{c}{d-\log c}.\notag
\end{equation}
\label{lem1}
\end{lemma}

\begin{proof}
Let $c$ be any positive integer in $[2^{d-2}]$.
Then all sets in $\mathcal{R}(2^{d-1}-c)$ are contained in $[d-1]$.
For any $A\subseteq [d-1]$,
write $A^c$ as the complement $[d-1]\backslash A$. 
Sine $A\bigtriangleup B=A^{c}\bigtriangleup B^{c}$,
we have $A\prec_{col}B$ if and only if $B^c\prec_{col} A^c$.
Therefore, we have
\begin{align}\label{equ:R+R^c}
2^{[d-1]}\setminus \mathcal{R}(2^{d-1}-c)=\{ [d-1]\setminus H : H\in \mathcal{R}(c) \}. 
\end{align}
Thus we can conclude that 
\[
W(2^{d-1}-c)=\frac{2^{d}-1}{d}-\sum_{R\in \mathcal{R}(c)}\frac{1}{d-|R|}\geqslant \frac{2^{d}-1}{d}-\frac{c}{d-\log c}, 
\]
where we use the property that any $R\in \mathcal{R}(c)$ satisfies $|R|\leq \log c$.
This finishes the proof.
\end{proof}

The original statement of the lemma below is also from \cite{PS2021}. However, as we require a slightly modified version for upcoming proofs, we provide the proof along with necessary adjustments here.

\begin{lemma}[\cite{PS2021}]\label{appoximation_lemma}
Let $d$ and $c$ be positive integers with $d\ge 4$ and $c\le 2^{d}$. Let $V$ be an arbitrary finite set with $|V|\ge d$. For any hereditary family $\mathcal{H}\subseteq 2^{V}$ with $|\mathcal{H}|\ge 2^{d}-c$, we have 
\begin{equation}
\sum_{H\in \mathcal{H}}\frac{1}{|H|+1}\ge W(2^{d}-c).\notag
\end{equation}
Moreover, if there are $e$ non-isolated vertices in $\mathcal{H}$ with $e>d$, then 
\begin{equation}
\sum_{H\in \mathcal{H}}\frac{1}{|H|+1}\ge W(2^{d}-c)+\frac{1}{6}(e-d).\notag
\end{equation}
\label{lem2}
\end{lemma}

\begin{proof}
Let $d$, $n$, $c$ and $\h$ be given as in the statement. By applying Theorem \ref{katona} with $f(k)=\frac{1}{k+1}$ for $k\in\mathbb{N}$, it is easy to get the first part of this lemma
\[
\sum_{H\in \mathcal{H}}\frac{1}{|H|+1}\ge \sum_{R\in \mathcal{R}(2^{d}-c)}\frac{1}{|R|+1}=W(2^{d}-c). 
\]

In order to prove the second part we need some preparation. For $i\in \mathbb{N}$, let $h_{i}$ be the number of $i$-element sets in $\h$ and $r_{i}$ be the number of $i$-element sets in $\mathcal{R}(2^{d}-c)$. 
Note that we have $r_{1}\le d$ and $r_{i}=0$ for $i\ge d$. For any $s\in [d]_{0}$, we set $g(k)=1$ for $k\le s$ and $g(k)=0$ for $k>s$. 
Then by applying Theorem \ref{katona} with $f=g$,
it yields that 
\begin{equation}
\sum_{i\in [s]_{0}}h_{i}\ge \sum_{i\in [s]_{0}}r_{i}.  \notag
\end{equation}

Let $H_{1},...,H_{|\h|}$ be an enumeration of $\h$ such that $|H_{j}|\le |H_{j+1}|$ for any $j\in [|\h|-1]$. For any $i\in [d-1]$, let $\phi (i)$ denote the number of edges of size at most $i$ in the family $\mathcal{R}(2^{d}-c)$, that is $\phi (i)=\sum_{j\in [i]_{0}}r_{j}$. 
Let $\h_{0}=\{H_{1}\}=\{\emptyset\}$. 
For any $i\in [d-1]$, we set $\h_{i}=\{H_{\phi(i-1)+1},...,H_{\phi(i)}\}$. 
Note that $|\h_{i}|=r_{i}$ for $i\in [d-1]_{0}$. 
The inequality $\sum_{i\in [s]_{0}}h_{i}\ge \sum_{i\in [s]_{0}}r_{i}$ implies that for $H\in \h_{i}$, where $i\in [d-1]_{0}$, we have $|H|\le i$. 
Since $r_{1}\le d$, we know that $H_{i}\notin \h_{1}$ for $i>d+1$. On the other hand, by assumption, there are $e$ non-isolated vertices in $\mathcal{H}$ with $e>d$, so we have $|H_{d+2}|=...=|H_{e+1}|=1$ but $H_{d+2},...,H_{e+1}$ are not in $\h_{1}$. Thus we have 
$$
\sum_{H\in \mathcal{H}}\frac{1}{|H|+1}\ge \sum_{i\in [d-1]_{0}}\sum_{H\in \h_{i}}\frac{1}{|H|+1}\ge \left(\frac{1}{2}-\frac{1}{3}\right)(e-d)+\sum_{i\in [d-1]_{0}}\frac{r_{i}}{i+1}=W(2^{d}-c)+\frac{1}{6}(e-d),
$$
finishing the proof.
\end{proof}

\section{Proof of Theorem~\ref{zhudingli_d=5}}

In this section, we present the proof of Theorem~\ref{zhudingli_d=5}, i.e., to show $m(n,11)= 5.3n$ for $10 \mid n$.
By Construction \ref{construction_c=d}, we have already seen that $m(n,11)\le 5.3n$ holds for $10 \mid n$. 
It remains to prove that $m(n,11)\ge 5.3n$ for $10 \mid n$. 
Using Proposition~\ref{dengjiatiaojian}, it suffices to show that for $10 \mid n$, a hereditary family $\mathcal{F}\subseteq 2^{[n]}$ with $\delta(\mathcal{F})\geqslant 12$ must satisfy $|\mathcal{F}|\geq 5.3n+1$. 
We will prove the following stronger statement, 
which holds for general $n\in \mathbb{N}^{+}$ and characterizes the unique extremal family.
 
\begin{thm}\label{thm000}
For any $n\in \mathbb{N}^{+}$, let $V$ be an $n$-element set, and $\mathcal{F} \subseteq 2^{V}$ be a hereditary family with $\delta (\mathcal{F} )\ge 12$. 
If $|\F|\le 5.3n+1$, then we have $10 \mid n$, $|\F|=5.3n+1$ and $\F\cong \mathcal{F}_{0}(n,5)$. 
\end{thm}

Recall that for $d\geq 5$, the family $\mathcal{F}_{0}(n,d)$ are given in Construction~\ref{construction_c=d}.

\subsection{Preparetion}
In this subsection, we establish the necessary notations and introduce a key lemma for the proof of Theorem \ref{thm000}.
Throughout the remaining of this section, let $V$ be a set with $|V|=n$, and $\mathcal{F} \subseteq 2^{V}$ be a hereditary family with $\delta(\mathcal{F})\ge 12$.
We may further assume that $\mathcal{F}$ is {\it minimal} in the sense that 
for any maximal set $F\in \mathcal{F}$, $\delta(\mathcal{F}\setminus \{F\})\le 11$.\footnote{A set $F\in \mathcal{F}$ is called {\it maximal} if there is no other set $F^\prime \in \mathcal{F}$ with $F\subsetneq F^\prime$. It is evident that for a hereditary family $\mathcal{F}$ and a maximal set $F\in \mathcal{F}$, the subfamily $\mathcal{F}\setminus \{F\}$ remains hereditary.}

Now, we will prove the first lemma, which asserts that under the aforementioned condition, every set $F$ has size at most $4$.

\begin{lemma}\label{lem:|F|<=4}
  Let $V$ be a set with $|V|=n$, and $\mathcal{F} \subseteq 2^{V}$ be a minimal hereditary family with $\delta(\mathcal{F})\ge 12$. 
  Then every set $F\in \mathcal{F}$ satisfies $|F|\leq 4$.
\end{lemma}

\begin{proof}
Suppose not. Then there exists a maximal set $F_{0}\in \mathcal{F}$ with $|F_{0}|\geqslant 5$. 
Since $\mathcal{F}$ is hereditary, for any $x\in F_{0}$, we have $d_{\mathcal{F}}(x)\geqslant 16$ and $d_{\mathcal{F}\setminus \{F_{0}\}}(x)\geqslant 15$. 
Thus we can derive that $\mathcal{F}\setminus \{F_{0}\}$ is still a hereditary family with minimum degree at least $12$, a contradiction. 
\end{proof}

Next, we proceed to define a weight function $\U:V\to \mathbb{R}^+$ such that $\sum_{x\in V} \U(x)=|\mathcal{F}|-1$. 
To do so, we first present the following algorithm to define a positive weight $\w(x,F)$ for all pairs $(x,F)$ satisfying $x\in F\in \mathcal{F}$.
Given $A\subseteq V$, we define the \emph{degree} of $A$ in $\mathcal{F}$ as follows: 
\begin{align}
d_{\mathcal{F}}(A)=|\{F\in \mathcal{F}:A\subseteq F\}|. \notag
\end{align}
In this notation, $d_{\mathcal{F}}(\{x\})=d_{\mathcal{F}}(x)$. 
Below is the algorithm for defining $\w(x,F)$.

\begin{algo}\label{alg}
    Fix $F\in \mathcal{F}$. If $|F|\ne 3$, we set $\w(x,F)=\frac{1}{|F|}$ for every $x\in F$. If $|F|=3$ and $F=\{x,y,z\}$, we perform the following process: 
\begin{enumerate}
\item[(1).] If there exists $G\in \mathcal{F}$ with $|G|=4$ such that $F\subseteq G$, then we set $\w(x,F)=\w(y,F)=\w(z,F)=\frac{1}{3}$; otherwise, go to the step (2).
\item[(2).] Without loss of generality, we assume that $d_{\F}(\{x,y\})\le d_{\F}(\{y,z\})\le d_{\F}(\{x,z\})$. 
If $d_{\F}(\{x,y\})\le d_{\F}(\{y,z\})\le 4<d_{\F}(\{x,z\})$, then we set $\w(x,F)=\w(z,F)=\frac{7}{20}$ and $\w(y,F)=\frac{6}{20}$. 
If $d_{\F}(\{x,y\})\le 4<d_{\F}(\{y,z\})\le d_{\F}(\{x,z\})$, then we set $\w(x,F)=\w(y,F)=\frac{7}{20}$ and $\w(z,F)=\frac{6}{20}$. Otherwise, we set $\w(x,F)=\w(y,F)=\w(z,F)=\frac{1}{3}$. 
\end{enumerate}
\end{algo}
It is evident that our definition of $\w(x,F)$ is well-defined, and for each $F\in \mathcal{F}$, we have $\sum_{x\in F} \w(x,F)=1$.
Now we define the \emph{weight} of $x$ as follows.

\begin{definition}
For each $x\in V$, let $\U(x)=\sum_{x\in F\in\mathcal{F} }\w(x,F).$
\end{definition}
\noindent It is easy to derive the following desired equation
\begin{align}\label{equ:sumU}
|\mathcal{F}|-1=|\mathcal{F}\setminus \{\emptyset\}|=\sum_{\emptyset \ne F\in\mathcal{F} }  1=\sum_{\emptyset \ne F\in\mathcal{F} } \sum_{x\in F} \w(x,F)=\sum_{x\in V}\sum_{x\in F\in\mathcal{F} }\w(x,F)=\sum_{x\in V}\U(x).
\end{align}

With this in mind, the problem now can be reduced to estimate the weight $\U(x)$.
To do so, we will analyze the structure of the links of vertices. 
For convenience, we introduce the following notation. 
For any $x\in V$, let $\ff_{i}(x)$ be the number of $i$-element sets in $\mathcal{F}(x)$. In particular, we have $\ff_{0}(x)=1$. 
Note that for $i\ge 4$, by Lemma~\ref{lem:|F|<=4}, we have $\ff_{i}(x)=0$ for any $x\in V$. 
Therefore, for any $x\in V$ we have $d_{\F}(x)=|\F(x)|=\ff_{0}(x)+\ff_{1}(x)+\ff_{2}(x)+\ff_{3}(x)$. 

Next, we introduce a crucial concept called \emph{\miniweight} vertices, which will play a significant role in the proofs. 
An important property that will be demonstrated later (see Lemma~\ref{lem:keyd=5}) is that a vertex is \miniweight \ if and only if its weight $\U(x)$ is less than $5.3$. 

\begin{definition}
  We say that a vertex $x\in V$ is \emph{\miniweight}, if $\ff_{1}(x)=4$, $\ff_{2}(x)=5$, and $\ff_{3}(x)=2$.   
\end{definition}

Before delving deeper into the properties on \miniweight \ vertices, we first present the following straightforward lemma.
Let $\Q=\{\q\in \mathcal{F}:|\q|=4\}$ denote the set of all $4$-element sets in $\mathcal{F}$. 
For $x\in V$, let $\Q(x)=\{\q\in \Q:x\in \q\}$.  
Note that $|\Q(x)|=\ff_3(x)$.

\begin{lem}
Let $x\in V$ be a \miniweight \ vertex. Then we have 
\begin{enumerate}[label =(\arabic*)]
\item $|\Q(x)|=2$. 
\item If $\Q(x)=\{\q_{1},\q_{2}\}$, then $|\q_{1}\cap \q_{2}|=3$ and $\mathcal{F}(x)=\{F\setminus \{x\}:x\in F\subseteq \q_{i}\mbox{ for some }i\in\{1,2\}\}$. 
\item $\U(x)=5.3-\frac{2}{15}$. 
\end{enumerate}
\label{lem100}
\end{lem}
\begin{proof}
(1) It is easy to see that we have $|\Q(x)|=\ff_{3}(x)=2$. 

(2) Because $\q_{1}\cup \q_{2}\subseteq N(x)$, we have $5\le |\q_{1}\cup \q_{2}|\le |N(x)|=\ff_{1}(x)+1=5$. Thus we have $|\q_{1}\cup \q_{2}|=5$ and $|\q_{1}\cap \q_{2}|=3$. It is obvious that $\{F\setminus \{x\}:x\in F\subseteq \q_{i}\mbox{ for some }i\in\{1,2\}\}\subseteq\mathcal{F}(x)$ because $\mathcal{F}(x)$ is a hereditary family. 
Since $|\mathcal{F}(x)|=\ff_{0}(x)+\ff_{1}(x)+\ff_{2}(x)+\ff_{3}(x)=12=|\{F\setminus \{x\}:x\in F\subseteq \q_{i}\mbox{ for some }i\in\{1,2\}\}|$, we derive that $\mathcal{F}(x)=\{F\setminus \{x\}:x\in F\subseteq \q_{i}\mbox{ for some }i\in\{1,2\}\}$. 

(3) Since $\mathcal{F}(x)=\{F\setminus \{x\}:x\in F\subseteq \q_{i}\mbox{ for some }i\in\{1,2\}\}$, then we know that for any $F\in\F(x)$ with $|F|=2$, we have $\w(x,F\cup\{x\})=\frac{1}{3}$. Thus we have $\U(x)=1+4\times\frac{1}{2}+5\times\frac{1}{3}+2\times\frac{1}{4}=5.3-\frac{2}{15}$. 
\end{proof}

\begin{corollary}
Let $\q\in \Q$. Then $\q$ contains at most three \miniweight \ vertices. 
\end{corollary}

\begin{proof}
If not, all vertices in $\q$ is \miniweight. Pick $x\in \q$, by Lemma \ref{lem100}, we can assume $\Q(x)=\{\q,\q'\}$ and $|\q\cap\q'|=3$. Assume $y\in\q\setminus\q'$ and $\Q(y)=\{\q,\q''\}$, we also have $|\q\cap\q''|=3$. Thus we know that $\q'\ne\q''$ and $\q\cap\q'\cap\q''\ne\emptyset$. Pick $z\in\q\cap\q'\cap\q''\subseteq\q$, we know that $\ff_{3}(z)\ge3$, so $z$ can not be \miniweight. This is contradicted with assumption. 
\end{proof}

In what follows, we aim to define a perturbation $\e: V\to \mathbb{R}$ such that $\sum_{x\in V} \e(x)=0$, and show that $\U(x)\geq 5.3+\e(x)$ holds for each $x\in V$.
Note that this would imply that $|\mathcal{F}|\geq 5.3n+1$.
For $x\in \q\in \Q$, let $\C(\q)=|\{y\in\q:y\mbox{ is \miniweight}\}|$. 
Then we define $\e(x,\q)$ as follows:
\[
\e(x,\q)=\begin{cases}
         -\frac{1}{15}, &\mbox{ if } x\mbox{ is \miniweight}, \\
         \frac{\C(\q)}{15(4-\C(\q))}, &\mbox{ if } x\mbox{ is not \miniweight}. 
         \end{cases} 
\]

\begin{definition}
    For any $x\in V$, we define $\e(x)=\sum_{\q\in\Q(x)}\e(x,\q). $
\end{definition}

Since $\C(\q)\le3$, our definition is well-defined. In particular, $\e(x)=0$ if $\ff_{3}(x)=0$. Moreover, the following fact shows that the sum of $\e(x)$ over all $x\in V$ indeed is zero:
\begin{align}\label{equ:sum-e}
\sum_{x\in V}\e(x)=\sum_{\q\in \Q}\sum_{x\in\q}\e(x,\q)=\sum_{\q\in \Q}\left(-\frac{\C(\q)}{15}+\frac{\C(\q)}{15(4-\C(\q))}(4-\C(\q))\right)=0.
\end{align}

Using \eqref{equ:sumU} and \eqref{equ:sum-e}, it is evident that if one can prove $\U(x)\ge 5.3+\e(x)$ for any $x\in V$, then it follows that $|\mathcal{F}|\geq 5.3n+1$. 
However, determining the structure of the (unique) extremal family requires additional efforts and considerations.
All the crucial properties required for proving Theorem \ref{thm000} will be demonstrated in the following key lemma. 

\begin{lem}[key lemma]\label{lem:keyd=5}
Let $V$ be an $n$-element set and $\mathcal{F} \subseteq 2^{V}$ be a minimal hereditary family with $\delta (\mathcal{F} )\ge 12$. Then we have 
\begin{enumerate}[label =(\arabic*)]
\item For any $x\in V$, we have $\U(x)\ge 5.3+\e(x)$. 
\item $\U(x)<5.3$ if and only if $x$ is \miniweight. 
\item If $\U(x)>5.3$, then $\U(x)>5.3+\e(x)$. 
\item If $\U(x)=5.3$, then either $\ff_{1}(x)=5,\ff_{2}(x)=6,\ff_{3}(x)=0$ or $\ff_{1}(x)=4,\ff_{2}(x)=6,\ff_{3}(x)=1$. 
\end{enumerate}
\label{keykey}
\end{lem}

The proof of this lemma will be provided in the subsequent subsection.

\subsection{Proof of Lemma~\ref{keykey}}
This subsection is devoted to the proof of Lemma \ref{keykey}. 
To accomplish this, we need to carefully estimate the weight $\U(x)$ for each vertex $x\in V$ based on the structure of its link $\mathcal{F}(x)$. 
In the following two lemmas, we present two special cases that will play significant roles in classifying the extremal families.

\begin{lem}
Let $V$ be an $n$-element set and $\mathcal{F} \subseteq 2^{V}$ be a minimal hereditary family with $\delta (\mathcal{F} )\ge 12$. 
Let $x\in V$. If $\ff_{3}(x)=0$, then $\U(x)\ge 5.3$. 
Moreover, the equality holds if and only if $\ff_{1}(x)=5$, $\ff_{2}(x)=6$ and $\w(x,F\cup\{x\})=\frac{6}{20}$ for any $F\in\mathcal{F}(x)$ with $|F|=2$. 
\label{lem102}
\end{lem}

\begin{proof}
Assume $\ff_{3}(x)=0$. We first claim that $\ff_{1}(x)\ge 5$. 
Otherwise $\ff_{1}(x)\le 4$, which leads to $\ff_{2}(x)\le 6$. 
Then $d_{\mathcal{F}}(x)=\ff_{0}(x)+\ff_{1}(x)+\ff_{2}(x)+\ff_{3}(x)\le 11$ which contradicts with $\delta(\mathcal{F})\ge 12$. 
Thus using Algorithm~\ref{alg}, we can derive that 
\[
\U(x)\ge 1+\frac{\ff_{1}(x)}{2}+\frac{6\ff_{2}(x)}{20} 
\ge 1+\frac{\ff_{1}(x)}{2}+\frac{6(11-\ff_{1}(x))}{20} 
\ge 1+\frac{5}{2}+\frac{6\times 6}{20}=5.3.
\]
It is clear that $\U(x)=5.3$ holds if and only if all equations hold in the above inequality, which would imply the ``moreover" part of the lemma. 
\end{proof}

\begin{lem}
Let $V$ be an $n$-element set and $\mathcal{F} \subseteq 2^{V}$ be a minimal hereditary family with $\delta (\mathcal{F} )\ge 12$. 
Let $x\in V$. 
Assume that $\ff_{3}(x)=1$ and $\ff_{1}(x)=4$. 
Then the following hold: 
\begin{enumerate}[label =(\arabic*)]
\item Let $\q=\{x,y_{1},y_{2},y_{3}\}\in\Q(x)$ and $N(x)=\{x,y_{1},y_{2},y_{3},z\}$. Then $\mathcal{F}(x)=2^{\{y_{1},y_{2},y_{3}\}}\cup\{\{z\},\{y_{i},z\}:i\in [3]\}$. In particular, $N(x)\subseteq N(z)$ and $N(x)\subseteq N(y_{i})$ for $i\in [3]$. 
\item $\C(\q)=0$. 
\item $\U(x)=5.3$. 
\end{enumerate}
\label{lem105}
\end{lem}

\begin{proof}
(1) Since $\ff_{3}(x)=1$, we have $\mathcal{F}(x)\subseteq 2^{\{y_{1},y_{2},y_{3},z\}}\setminus \{F\cup \{z\}:F\subseteq \{y_{1},y_{2},y_{3}\},|F|\ge 2\}=2^{\{y_{1},y_{2},y_{3}\}}\cup\{\{z\},\{y_{i},z\}:i\in [3]\}$. On the other hand, we have $|\mathcal{F}(x)|\ge 12=|2^{\{y_{1},y_{2},y_{3},z\}}\setminus \{F\cup \{z\}:F\subseteq \{y_{1},y_{2},y_{3}\},|F|\ge 2\}|$. So we have $\mathcal{F}(x)=2^{\{y_{1},y_{2},y_{3}\}}\cup\{\{z\},\{y_{i},z\}:i\in [3]\}$. 

(2) For any $i\in [3]$, because $\ff_{3}(x)=1$, there is no 4-element set in $\mathcal{F}$ containing $\{x,y_{i},z\}$. Then we know that for any $i\in [3]$, $y_{i}$ is not \miniweight. So we have $\C(\q)=0$. 

(3) For any $i\in [3]$, since $d_{\F}(x,y_{i})=5$ and $d_{\F}(x,z)=4$, we have $\w(x,\{x,y_{i},z\})=\frac{7}{20}$. Thus we have $\U(x)=1+4\times\frac{1}{2}+3\times\frac{7}{20}+3\times\frac{1}{3}+\frac{1}{4}=5.3$. 
\end{proof}

Now we are ready to give the proof of Lemma \ref{keykey}. 

\begin{proof}[\bf Proof of Lemma \ref{keykey}.]
Let $V$ be an $n$-element set and $\mathcal{F} \subseteq 2^{V}$ be a minimal hereditary family with $\delta (\mathcal{F} )\ge 12$. 
Pick arbitrary $x\in V$. 
We will prove all four items of this lemma simultaneously by analyzing all possible structures of the link $\mathcal{F}(x)$, i.e., the values of $\ff_i(x)$ for $0\leq i\leq 3$.

First, we assume that $x$ is \miniweight. 
By Lemma \ref{lem100} (1) and (3), we know that $|\Q(x)|=2$ and $\U(x)=5.3-\frac{2}{15}$. Moreover, for any $\q\in\Q(x)$ we have $\e(x,\q)=-\frac{1}{15}$. 
Combining these facts, we can derive easily that $\U(x)=5.3-\frac{2}{15}=5.3+\sum_{\q\in\Q(x)}\e(x,\q)=5.3+\e(x)$. 
This confirms all four items of this lemma in the case of \miniweight \ vertices.

Assume that $\ff_{3}(x)\ge 3$. 
In this case, we will show that $\U(x)>5.3+\e(x)$ where $\e(x)\geq 0$.\footnote{Here it is clear that $x$ is not \miniweight \ and thus by definition, $\e(x)\geq 0$.}
This is sufficient to establish all four items of this lemma in this particular case.
Since $\ff_{3}(x)\ge 3$, then we have $\ff_{1}(x)\ge 4$ and $\ff_{2}(x)\ge 6$. Moreover, $|\bigcup_{F\in\F(x),|F|=3}\binom{F}{2}|\ge 6$. 
If $\ff_{1}(x)\ge 5$, then we have $
\U(x)\ge 1+5\times\frac{1}{2}+6\times\frac{1}{3}+\frac{\ff_{3}(x)}{4} 
=5.5+\frac{\ff_{3}(x)}{4} 
>5.3+\frac{\ff_{3}(x)}{5} 
\ge 5.3+\e(x)$. 
If $\ff_{1}(x)=4$, then $3\le \ff_{3}(x)\le 4$. 
For any $\q\in\Q(x)$, we can take two different sets $P_{1},P_{2}\in\Q(x)\setminus\{\q\}$ such that $|\q\cap P_{1}\cap P_{2}|=1$. 
Assume $\{z\}=\q\cap P_{1}\cap P_{2}$. 
We have $\ff_{3}(z)\ge 3$, so $z$ is not \miniweight. 
This shows that $\C(\q)\le 2$ for any $\q\in\Q(x)$, thus implying that $\e(x,\q)\le \frac{1}{15}$ and $0\leq \e(x)\le \frac{\ff_{3}(x)}{15}$. 
Therefore, $\U(x)\ge 1+4\times\frac{1}{2}+6\times\frac{1}{3}+\frac{\ff_{3}(x)}{4} 
=5.3+(\frac{\ff_{3}(x)}{4}-\frac{3}{10}) 
>5.3+\frac{\ff_{3}(x)}{15} 
\ge 5.3+\e(x)$. 

Assume that $\ff_{3}(x)=0$. Then $\e(x)=0$. 
By Lemma \ref{lem102}, we see that $\U(x)\ge 5.3=5.3+\e(x)$ and $\U(x)=5.3$ only occurs when $\ff_{1}(x)=5$ and $\ff_{2}(x)=6$. 
Obviously in this case, if $\U(x)>5.3$, then $\U(x)>5.3+\e(x)$. 
So all four items of this lemma hold in this case.

Assume that $\ff_{3}(x)=1$.
Then $\ff_{1}(x)\ge 4$. 
We now show that if $\ff_{1}(x)\ge 5$, then $\U(x)>5.3+\e(x)$ where $\e(x)\geq 0$.
To see this, we first consider when $\ff_{1}(x)\ge 6$, which would imply  
that $\U(x)\ge 1+6\times\frac{1}{2}+\frac{6}{20}+3\times\frac{1}{3}+\frac{1}{4}=5.55>5.3+\frac{1}{5}\ge 5.3+\e(x)$. 
Now we may assume that $\ff_{1}(x)=5$ and thus $\ff_2(x)\geq 5$.
If $\ff_{2}(x)\ge 6$, then we have $\U(x)\ge 1+5\times\frac{1}{2}+3\times\frac{6}{20}+3\times\frac{1}{3}+\frac{1}{4}=5.65>5.3+\frac{1}{5}\ge 5.3+\e(x)$. 
So we may assume that $\ff_{1}(x)=\ff_{2}(x)=5$. 
Let $\q\in\Q(x)$ and $R=N(x)\setminus\q$, then $|R|=2$. We claim that there exist $y\in \q$ and $z\in R$, such that $\{y,z\}\in\F(x)$. Otherwise, $\ff_{2}(x)\le \binom{3}{2}+\binom{2}{2}=4<5$, which is a contradiction. 
Thus, we can get $\{x,y,z\}\in\F$ with $y\in \q$ and $z\in R$. 
Since $\ff_{3}(x)=1$, there is no 4-element set in $\F$ containing $\{x,y,z\}$. 
By Lemma \ref{lem100} (2), we know that $y\in\q$ can not be \miniweight, so $\C(\q)\le 2$ and thus $\e(x)\le \frac{1}{15}$. 
Since $d_{\F}(\{x,y\})\ge 5$ and $d_{\F}(\{x,z\})\le 3$, we also have $\w(x,\{x,y,z\})=\frac{7}{20}$. 
Finally we have $\U(x)\ge 1+5\times\frac{1}{2}+\frac{6}{20}+\frac{7}{20}+3\times\frac{1}{3}+\frac{1}{4}=5.4>5.3+\frac{1}{15}\ge 5.3+\e(x)$. 

Now consider that $\ff_{3}(x)=1$ and $\ff_{1}(x)=4$. 
Note that in this case, we can derive $\ff_{2}(x)=6$.
By Lemma \ref{lem105} (2) and (3), we have $\e(x)=0$ and $\U(x)=5.3=5.3+\e(x)$.  
So all four items of this lemma hold.

From now on, we may assume that $x$ is not a \miniweight \ vertex with $\ff_{3}(x)=2$. 
Clearly we have $\ff_{1}(x)\ge 4$.  
We make the following claim. 

\begin{claim}\label{claim-1}
Assume that $\ff_{3}(x)=2$ and $\Q(x)=\{\q_{1},\q_{2}\}$. If $x$ is not \miniweight\ and $|\q_{1}\cap\q_{2}|=3$, then $\C(\q_{i})\le 2$ for $i\in [2]$. 
\label{lem10225}
\end{claim}
\begin{proof}[Proof of Claim~\ref{claim-1}.]
Assume that $\q_{1}=\{x,y_{1},y_{2},y_{3}\}$ and $\q_{2}=\{x,y_{2},y_{3},y_{4}\}$. 
If there exists $i_{0}\in [2]$ such that $\C(\q_{i_{0}})=3$ (say $i_0=1$), 
then $y_{j}$ is \miniweight\ for $j\in [3]$. 
As $y_{1}$ is \miniweight, there exists $\q'\in \Q(y_{1})$ such that $|\q'\cap\q_{1}|=3$. 
Then $\q'\cap\{y_{2},y_{3}\}\ne\emptyset$, 
so we may assume $y_{2}\in\q'$. 
Since $\q'\ne \q_{1},\q_{2}$, we have $\ff_{3}(y_2)\ge 3$. 
But $y_{2}$ is also \miniweight, so $\ff_{3}(y_2)=2$, a contradiction. 
This shows that $\C(\q_{i})\le 2$ for $i\in [2]$. 
\end{proof}

Suppose that $\ff_{3}(x)=2$ and $\ff_{1}(x)\ge 5$. 
In this case, we will show that $\U(x)>5.3+\e(x)$ where $\e(x)\geq 0$. 
Assume that $\Q(x)=\{\q_{1},\q_{2}\}$. 
If $|\q_{1}\cap\q_{2}|=3$, by Claim \ref{lem10225}, we see that $\C(\q_{i})\le 2$ for $i\in [2]$, which implies that $\e(x)\le \frac{2}{15}$. 
Then we have $\U(x)\ge 1+5\times\frac{1}{2}+5\times\frac{1}{3}+\frac{2}{4}=5.3+\frac{11}{30}>5.3+\frac{2}{15}\ge 5.3+\e(x)$. 
Now let $|\q_{1}\cap\q_{2}|\le 2$. 
Then clearly $|\binom{\q_{1}\setminus\{x\}}{2}\cup\binom{\q_{2}\setminus\{x\}}{2}|\ge 6$, 
and in this case, we can still derive that $\U(x)\ge 1+5\times\frac{1}{2}+6\times\frac{1}{3}+\frac{2}{4}=6>5.3+\frac{2}{5}\ge 5.3+\e(x)$. 

It remains to consider the final case when $\ff_{3}(x)=2$ and $\ff_{1}(x)=4$.
In this case, again we will show that $\U(x)>5.3+\e(x)$. 
Since $x$ is not \miniweight, we can derive that $\ff_{2}(x)=6$. 
Assume that $\Q(x)=\{\{x,y_{1},y_{2},y_{3}\},\{x,y_{2},y_{3},y_{4}\}\}$.
This implies that $\{x,y_{1},y_{4}\}\in \F$. 
By Claim \ref{lem10225}, we can infer that $\e(x)\le\frac{2}{15}$. 
Hence, $\U(x)\ge 1+4\times\frac{1}{2}+\frac{6}{20}+5\times\frac{1}{3}+2\times\frac{1}{4}=5.3+\frac{1}{6}>5.3+\frac{2}{15}\ge 5.3+\e(x)$. 

The discussion above encompasses all possible values of $\ff_i(x)$ for $0\leq i\leq 3$,
where in each case, we have established that all four items of Lemma \ref{keykey} hold.
\end{proof}

\subsection{Completing the proof of Theorem \ref{thm000}}
In this subsection, we prove Theorem \ref{thm000}. 
Let $V$ be an $n$-element set and $\mathcal{F} \subseteq 2^{V}$ be a minimal hereditary family with $\delta (\mathcal{F} )\ge 12$ and $|\F|\le 5.3n+1$. 
We want to show that $10 \mid n$, $|\F|=5.3n+1$ and $\F\cong \mathcal{F}_{0}(n,5)$. 

By Lemma \ref{keykey} (1), for any $x\in V$ we have $\U(x)\ge 5.3+\e(x)$. 
Using \eqref{equ:sumU} and \eqref{equ:sum-e}, we can get 
\[
|\F|-1=\sum_{x\in V}\U(x)\ge \sum_{x\in V}(5.3+\e(x))=5.3n+\sum_{x\in V}\e(x)=5.3n.
\]
Thus we derive that $|\F|=5.3n+1$. Since $|\F|$ is an integer, we must have $10 \mid n$.

In what follows, we proceed to show $\F\cong \mathcal{F}_{0}(n,5)$.
We claim that $\U(x)=5.3$ for any $x\in V$. 
If not, since $\sum_{x\in V}\U(x)=5.3n$, there exists some $x'\in V$ with $\U(x')>5.3$. 
Then by Lemma \ref{keykey} (3), $\U(x')>5.3+\e(x')$, which leads to $\sum_{x\in V}\U(x)>5.3n$,  a contradiction.

Using Lemma \ref{keykey} (4), we see $\ff_{3}(x)\le 1$ holds for any $x\in V$. 
Thus every two distinct sets in $\Q$ are disjoint. 
Let $S:=V\setminus \bigcup_{\q\in\Q}\q$. Then we can get the following partition of $V$:
\[
V=S\cup \bigcup_{\q\in\Q}\q. 
\]

Let $\q\in\Q$. 
Since $\U(x)=5.3$ and $\ff_{3}(x)=1$ for any $x\in\q$, by Lemma \ref{lem105} (1), 
there exists a set $N_{\q}\subseteq V$ with $|N_{\q}|=5$ such that $N(x)=N_{\q}$ for any $x\in\q$. Note that $\q\subseteq N_\q$.
We set $s_{\q}$ to be the vertex in $N_{\q}\setminus \q$. 
Then by Lemma \ref{lem105} (1), we have $F\cup\{s_{\q}\}\in \F$ for any $F\in\binom{\q}{2}$. 
Now we claim that $s_{\q}\in S$. 
If not, then there exists $\q'\in\Q(s_{\q})$ with $\q\cap\q'=\emptyset$, from which we can derive that $\ff_{1}(s_{\q})\ge 7$. 
This is a contradiction to Lemma \ref{keykey} (4).
So we have $s_{\q}\in S$ for any $\q\in\Q$. 
Then we can define a map $\PHI:\Q \to S$ by letting $\PHI(Q)=s_Q$.

Next we will prove that $\PHI$ is bijective. 
First, we show $\PHI$ is injective. 
Suppose for a contradiction that there exist two different sets $\q,\q'\in\Q$ with $s_{\q}=s_{\q'}$. Then $N(s_{\q})\supseteq \q\cup\q'\cup\{s_{\q}\}$, implying that $\ff_{1}(s_{\q})\ge 8$. Again, this contradicts Lemma \ref{keykey} (4).

Now we show that $\PHI$ is surjective.
Let $x\in S$. We first claim that there exists some $y_{0}\in N(x)$ with $\ff_{3}(y_{0})=1$. 
If not, we have $\ff_{3}(y)=0$ for any $y\in N(x)$. 
Take any 2-set $\{y,z\}\in \F(x)$. 
By Lemma \ref{lem102}, we know that $\w(x,\{x,y,z\})=\frac{6}{20}$. 
Thus we get $\w(y,\{x,y,z\})=\frac{7}{20}$. 
By Lemma \ref{lem102}, we have $\U(y)>5.3$, which contradicts that $\U(y)=5.3$. 
Thus, there exists some $y_{0}\in N(x)$ with $\ff_{3}(y_{0})=1$. 
We take $\q_{0}\in\Q(y_{0})$. Then $x\in N(y_{0})=N_{\q_{0}}$, so we have $x=s_{\q_{0}}$. This tells us that $\PHI$ is surjective. We have proved that $\PHI$ is bijective.

Take any $s_{\q}\in S$. 
Then we have $N(s_{\q})\supseteq \q\cup\{s_{\q}\}$ and $|N(s_{\q})|=6$. 
Since $\PHI$ is a bijective mapping, the unique vertex in $N(s_{\q})\setminus (\q\cup \{s_{\q}\})$ must also be in $S$. 
Let us denote this vertex as $s_{\SIG(\q)}$. 
As a result, $s_{\q}$ and $s_{\SIG(\q)}$ are paired up.
From the above analysis, we can conclude that $\SIG$ is a mapping from $\Q$ to itself, satisfying the properties that $\SIG(\q)\ne \q$ and $\SIG^{2}(\q)=\q$ for any $\q\in\Q$. 
We observe that $\{s_{\q},s_{\SIG(\q)}\}\in \F$.
Let 
\[
\F_{0}=\Q \cup \{\{s_{\q},s_{\SIG(\q)}\}:\q\in\Q\} \cup \bigcup_{\q\in\Q}\{F\subseteq V:F\subseteq \q\cup\{s_{\q}\} \mbox{ and } |F|\le 3\}.
\]
Note that $\F_{0}\cong \mathcal{F}_{0}(n,5)$.
Our previous discussion implies that $\F\supseteq \F_{0}$. 
Since we also have that $|\F|=5.3n+1=|\F_{0}|$, 
it follows that $\F=\F_{0}\cong \mathcal{F}_{0}(n,5)$,
completing the proof of Theorem \ref{thm000}.
\qed

\section{Proof of Theorem~\ref{zhudingli}: reducing to Theorem~\ref{auxiliary_theorem}}

In this section, we provide a proof of Theorem \ref{zhudingli} by assuming Theorem \ref{auxiliary_theorem}. The upper bounds of Theorem \ref{zhudingli} have been established through Constructions \ref{construction_c<d} and \ref{construction_c=d}. Now, we focus on proving the corresponding lower bounds. 
To facilitate this, we first formalize an equivalent and unified statement (i.e., Theorem \ref{combined_theorem}) that encompasses both cases of Theorem \ref{zhudingli}.

We give the following definitions. 

\begin{definition} For $d,c\in\mathbb{N}$ with $1\leq c\leq d$, let
    $$\wb= \begin{cases}\frac{2^d-c}{d} & \text { for } 1\leq c\leq d-1, \\ \frac{2^d-d-\frac{1}{2}}{d} & \text { for } c=d.\end{cases}$$
\end{definition}

\begin{definition}
    For a hereditary family $ \g \subseteq 2^{[n]}$ and $x\in [n]$, we define the \emph{weight} of $x$ in $\g$ as $$\ww_\g(x)=\sum_{x\in F\in\F}\frac{1}{|F|}=\sum_{H\in \F(x)}\frac{1}{|H|+1}.$$ 
\end{definition}

Our proof is based on a standard approach as outlined in Section~2 that involves utilizing the concept of ``weight'' for vertices.
We often omit the subscript of $\ww_\g(x)$ and just call it the weight of $x$ or $\ww(x)$. 
Note that we have
\begin{align}
\sum_{x\in [n]}\ww_\g(x)=|\F\setminus \{\emptyset\}|=|\F|-1.
\end{align}

By utilizing the aforementioned definitions, explanations, and Proposition~\ref{dengjiatiaojian}, we observe that Theorem \ref{zhudingli} can be equivalently stated as follows.

\begin{theorem}\label{combined_theorem}
    Given $n,d,c\in\mathbb{N}$ with $d\ge 50$ and $1\leq c\leq d$. For any hereditary family $\g \subseteq 2^{[n]}$ with $\md(\g)\geq 2^{d-1}-c+1$, we have 
$$\sum_{x\in [n]}\ww_\g(x)\geq \wb n.$$
\end{theorem}

It is worth noting that when $c=d$, we can demonstrate that the only hereditary family $\mathcal{F}$ satisfying $|\mathcal{F}|-1=\sum_{x\in [n]}\ww_\g(x)=\wb n=\frac{2^d-d-1/2}{d}n$ is the family $\mathcal{F}_0(n,d)$ defined in Construction \ref{construction_c=d} (for detailed explanations, see the remarks before Subsection~\ref{subsec:iso-pile}).

From now on, we always assume that $\mathcal{F}\subseteq 2^{[n]}$ is a hereditary family with $\delta(\mathcal{F})\geqslant 2^{d-1}-c+1$ where $d\geq 50$ and $c\in [d]$. 
It suffices to prove the average weight in $\g$ is at least $\wb$. 

\begin{definition}$\ $
A vertex $x\in [n]$ is called \emph{good} in $\mathcal{F}$ if $|N(x)|\geq d+1$ and \emph{bad} in $\mathcal{F}$ if $|N(x)|=d$.
\end{definition}

For any vertex $x\in [n]$, we have $|N(x)|\geq 1+ \lceil\log(\md(\g))  \rceil\geq 1+ \lceil\log(2^{d-1}-c+1) \rceil=d$. 
Therefore every vertex is either good or bad in $\mathcal{F}$.

The following lemma provides a lower bound on the weight of $x\in [n]$ based on its size. Specifically, it shows that if $x$ is good in $\mathcal{F}$, then its weight $\ww(x)>\wb$.

\begin{lemma}\label{weight_lemma}
Let $n$ and $d\geq 50$ be two positive integers and let $\F\subseteq 2^{[n]}$ be a hereditary family. Let $x\in [n]$ be a vertex with $d_{\F}(x)\geqslant 2^{d-1}-c+1$, where $c\in [d]$. Then the following hold:
    \begin{enumerate}
        \item if $|N(x)|=d$, then $\ww(x)> \wb-\frac{1}{18}$; and
        \item if $|N(x)|>d$, then $\ww(x)> \wb-\frac{1}{18}+\frac{|N(x)|-d}{6}> \wb$.
    \end{enumerate}
    
\end{lemma}
\begin{proof}
Since $|\F(x)|=d_{\F}(x)\geqslant 2^{d-1}-c+1$, by Lemma \ref{lem2}, we have 
\[
\ww(x)=\sum_{H\in \F(x)}\frac{1}{|H|+1}\geqslant W(2^{d-1}-c+1)+\frac{1}{6}(|N(x)|-d).
\]
So in both cases, it is sufficient for us to prove that $W(2^{d-1}-c+1)>\wb-\frac{1}{18}$ for $d\geqslant 50$ and $c\in [d]$. 
For convenience, we write $\mathcal{H}_{c}=2^{[d-1]}\setminus \mathcal{R}(2^{d-1}-c+1)$. 
Then for any $c\in [d]$, 
\[
\wb-W(2^{d-1}-c+1)\le \frac{2^{d}-c}{d}-W(2^{d-1}-c+1)=\sum_{H\in \mathcal{H}_{c}}\left(\frac{1}{|H|+1}-\frac{1}{d}\right)\leqslant \sum_{H\in \mathcal{H}_{d}}\left(\frac{1}{|H|+1}-\frac{1}{d}\right). 
\]
Using \eqref{equ:R+R^c}, we have $\mathcal{H}_{d}=\{ [d-1]\setminus H : H\in \mathcal{R}(d-1) \}$. 
Then it implies that 
\[
\wb-W(2^{d-1}-c+1)\leqslant \sum_{H\in \mathcal{H}_{d}}\left(\frac{1}{|H|+1}-\frac{1}{d}\right)\leqslant \sum_{H\in \mathcal{R}(d-1)}\left(\frac{1}{d-|H|}-\frac{1}{d}\right). 
\]
By the calculation in Appendix \ref{compute}, we can derive that whenever $d\ge 50$, 
\[\wb-W(2^{d-1}-c+1)\le \sum_{H\in \mathcal{R}(d-1)}\left(\frac{1}{d-|H|}-\frac{1}{d}\right)<\frac{1}{18}. \]
This shows that for $d\geqslant 50$ and $c\in [d]$, it holds that $W(2^{d-1}-c+1)>\wb-\frac{1}{18}$.  
\end{proof}

Before we proceed further, let us introduce some additional definitions that play a crucial role in the proof of Theorem \ref{combined_theorem}.

\begin{definition}
Let $P\subseteq [n]$ be a subset with $|P|=d$. We say that $P$ is a \emph{pile} of $\mathcal{F}$ if 
\begin{itemize}
   \item for any $y\in P$, we have $P\subseteq N(y)$, and
   \item there exists $z\in P$ such that $P=N(z)$. 
\end{itemize}
\end{definition}

\begin{definition}~
\begin{itemize}
\item A pile $P$ is an \emph{intersecting pile} if it intersects another pile. 
\item A pile $P$ is an \emph{isolated pile} if it does not intersect any other piles.\footnote{Piga and Schülke \cite{PS2021} employ a concept similar to our definition of piles, known as clusters. However, there are important distinctions between the two. One notable difference is that their definition of clusters allows for the presence of bad vertices outside of clusters. As a result, they need to redistribute at least $\frac12-\frac{c-1}{d-c}\geq \frac16$ weight to these bad vertices, which ultimately leads to the bound $c\leq \frac{d}{4}$. In contrast, our definition of piles ensures that there are no bad vertices outside of piles (as shown in Lemma~\ref{pile_lemma}).}
\end{itemize}
\end{definition}

According to this definition, we can partition the vertex set $[n]$ into three parts (also see \eqref{equ:3parts} in the proof of Theorem \ref{combined_theorem} below):
\begin{itemize}
\item The first part $J$ consists of vertices that are not contained in any pile; 
\item The second part $K$ consists of vertices that are contained in intersecting piles; and 
\item The third part consists of all remaining vertices that are contained in isolated piles. 
\end{itemize}
Our approach for Theorem \ref{combined_theorem} proceeds to demonstrate that the average weight in each part is at least $\wb$, respectively.

The upcoming lemma establishes a connection between bad vertices and piles. 
Combining this lemma with Lemma~\ref{weight_lemma}, we can conclude that the average weight in the first part $J$ is at least $\wb$.

\begin{lemma}\label{pile_lemma}
    For $d\geq 6$, every bad vertex $x$ is in exactly one pile, which is $N(x)$. In particular, this implies that every vertex in $J$ is a good vertex in $\mathcal{F}$.
\end{lemma}
\begin{proof}
Let $x$ be a bad vertex in $\mathcal{F}$.
If $x$ is contained in a pile $P$, then $P\subseteq N(x)$ and $d=|P|\leq |N(x)|=d$, implying that $P=N(x)$. 
So it suffices to prove that $N(x)$ is a pile, that is, 
for all $y \in N(x)$, it holds that $N(x)\subseteq N(y)$.
Suppose this is not the case. Then there exists $y\in N(x)$ such that $N(x)\not\subseteq N(y)$. So there exists $z\in N(x)\setminus \{x,y\}$ such that $z\not\in N(y)$. 
This implies that $\{y,z\}\not \in \g(x)$. 
Notice that $\g(x)$ is a hereditary family, so for any $F\supseteq \{y,z\}$, $F\notin \g(x)$. Thus we have $d_\g(x)\leq 2^{d-1}-2^{d-3}<2^{d-1}-d+1\leq \md(\g)$ when $d\geq 6$.
This is a contradiction.

The second conclusion follows directly from the definition of $J$.
\end{proof}

\begin{remark}
We would like to highlight that Lemma \ref{pile_lemma} is not applicable for $d=5$. 
This distinction is one of the reasons why we need separate proofs for Theorems \ref{zhudingli_d=5} and \ref{zhudingli}.
\end{remark}

The following result asserts that for any isolated pile $P$, the average weight in $P$ is at least $\wb$. 

\begin{theorem}\label{auxiliary_theorem}
    Let $n,d,c\in\mathbb{N}$ with $d\ge 50$ and $1\leq c\leq d$. For any hereditary family $ \g \subseteq 2^{[n]}$ with $\md(\g)\geq 2^{d-1}-c+1$. If $P\subseteq [n]$ is an isolated pile in $\g$, then
    $$\sum_{x\in P}\ww_\g(x)\geq \wb d.$$
\end{theorem}

We will postpone the proof of this result in Section \ref{sec5}. 
Now assuming Theorem~\ref{auxiliary_theorem}, we are ready to present the proof of Theorem \ref{combined_theorem}  (i.e., Theorem \ref{zhudingli}).

\begin{proof}[\bf Proof of Theorem \ref{combined_theorem} (assuming Theorem \ref{auxiliary_theorem}).]
Let $n$, $d$, $c$ and $\F$ be given as in Theorem \ref{combined_theorem}. Let $\mathcal{P}=\{P\subseteq [n]: P\mbox{ is a pile} \}$ be the set of all piles. For $x\in [n]$, let $\mathcal{P}(x)=\{P\in \mathcal{P}: x\in P\}$ be the set of piles containing $x$. 
In this notion, Lemma \ref{pile_lemma} says that every bad vertex $z\in [n]$ satisfies $|\mathcal{P}(z)|=1$. 
To get a partition of $[n]$, we consider the following partition of $\mathcal{P}$:
\begin{align}
\mathcal{P}_{1}&=\{P\in \mathcal{P}: P\text{ is an isolated pile in } \mathcal{F}\}, \notag \\
\mathcal{P}_{2}&=\{P\in \mathcal{P}: P\text{ is an intersecting pile in } \mathcal{F}\}. \notag 
\end{align}

Note that we have $\mathcal{P}=\mathcal{P}_{1}\cup \mathcal{P}_{2}$. 
Let $\G=\{x\in [n]: \mathcal{P}(x)=\emptyset \}$ and $K={\textstyle \bigcup_{P\in \mathcal{P}_{2} }}P$ Then we have the following partition of $[n]$:
\begin{equation}\label{equ:3parts}
[n]=\G\cup K\cup (\bigcup_{P\in \mathcal{P}_{1}}P).
\end{equation}

As Lemma \ref{pile_lemma} indicates, every vertex in $\G$ is good and by Lemma \ref{weight_lemma} (2), it has weight larger than $\wb$. 
So we have 
\begin{equation}\label{equ:J}
\sum_{x\in \G}\omega (x)> \sum_{x\in \G}\wb =\wb |\G|.
\end{equation}

Next, we prove that the average weight of vertices in $K$ is more than $\wb$. 
We distinguish between $K_{1}=\{x\in K:|\mathcal{P}(x)|=1\}$ and $K_{2}=K\setminus K_{1}$. 
By Lemma \ref{weight_lemma}, 
every vertex $z\in K_1$ satisfies $\ww(z)> \wb-\frac{1}{18}$. 
So it is sufficient to prove that $\sum_{x\in K_2}(\ww(x)-\wb)\ge \sum_{z\in K_1}\frac{1}{18}$.  
We have the following claim. 
For a pile $P$, let $\theta_{P}=|\{z\in P:|\mathcal{P}(z)|=1\}|$.

\begin{claim}\label{Claim:Px>1}
Every vertex $x\in K_2$ satisfies $\ww(x)> \wb +\frac{1}{18} \sum_{P\in \mathcal{P}(x)}\frac{\theta_{P}}{d-\theta_{P}}$. 
\end{claim}

\begin{proof}[Proof of Claim~\ref{Claim:Px>1}.]
Assume that $x\in [n]$ with $|\mathcal{P}(x)|\ge 2$. Then we have $|N(x)|\ge \sum_{P\in \mathcal{P}(x)}\theta_{P}+\max_{Q\in \mathcal{P}(x)}(d-\theta_{Q})=\sum_{P\in \mathcal{P}(x)}\theta_{P}+d-\min_{Q\in \mathcal{P}(x)}\theta_{Q}$. 
Then by Lemma \ref{weight_lemma} (2), we have 
\begin{align}
\omega(x)> \wb -\frac{1}{18} +\frac{1}{6}(|N(x)|-d) \ge \wb -\frac{1}{18} +\frac{1}{6}(\sum_{P\in \mathcal{P}(x)}\theta_{P}-\min_{Q\in \mathcal{P}(x)}\theta_{Q}) \ge \wb -\frac{1}{18} +\frac{1}{12}\sum_{P\in \mathcal{P}(x)}\theta_{P}. \notag 
\end{align}
Note that $\theta_{P}\geqslant 1$ for any $P\in \mathcal{P}$.
So we have $\sum_{P\in \mathcal{P}(x)}\theta_{P}\geqslant 2$. 
Since $-\frac{1}{18}+\frac{1}{12}\lambda \geqslant \frac{1}{18} \lambda$ for any $\lambda\ge 2$, 
we can derive that 
\[
\omega(x)
\ge \wb -\frac{1}{18} +\frac{1}{12}\sum_{P\in \mathcal{P}(x)}\theta_{P} \geqslant \wb +\frac{1}{18} \sum_{P\in \mathcal{P}(x)}\theta_{P} \ge \wb +\frac{1}{18} \sum_{P\in \mathcal{P}(x)}\frac{\theta_{P}}{d-\theta_{P}},
\]
finishing the proof of this claim.
\end{proof}

Using Claim \ref{Claim:Px>1}, it is straightforward to obtain the following inequality  
\begin{align}
\sum_{x\in K_{2}}\omega (x)&> \sum_{x\in K_{2}} \left(\wb +\frac{1}{18} \sum_{P\in \mathcal{P}(x)}\frac{\theta_{P}}{d-\theta_{P}}\right) =\wb |K_{2}|+\frac{1}{18} \sum_{x\in K_{2}}  \sum_{P\in \mathcal{P}(x)}\frac{\theta_{P}}{d-\theta_{P}} \notag \\
&=\wb |K_{2}|+\frac{1}{18} \sum_{P\in \mathcal{P}_{2}}  \sum_{x\in P\cap K_{2}}\frac{\theta_{P}}{d-\theta_{P}} =\wb |K_{2}|+\frac{1}{18} \sum_{P\in \mathcal{P}_{2}}  \theta_{P}. \notag 
\end{align}
Moreover, we have 
$\sum_{x\in K_{1}}\omega (x)> \sum_{x\in K_{1}}\left(\wb -\frac{1}{18}\right)=\wb |K_{1}|-\frac{1}{18} \sum_{P\in \mathcal{P}_{2}}  \theta_{P}.$
Adding the above two inequalities up, we can derive that
\begin{equation}\label{equ:K}
\sum_{x\in K}\omega (x)=\sum_{x\in K_{1}}\omega (x)+\sum_{x\in K_{2}}\omega (x)>\wb (|K_{1}|+|K_{2}|)= \wb |K|.
\end{equation}
Finally, putting \eqref{equ:J}, \eqref{equ:K} and Theorem \ref{auxiliary_theorem} all together,  
we can obtain that 
\begin{align}
\sum_{x\in [n]}\ww(x)&=\sum_{x\in \G}\ww(x)+\sum_{x\in K}\ww(x)+\sum_{P\in\mathcal{P}_{1}}\sum_{x\in P}\ww(x) \ge \wb(|\G|+|K|+d|\mathcal{P}_{1}|) =\wb n, \notag
\end{align}
finishing the proof of Theorem \ref{combined_theorem}.
\end{proof}

\section{Proof of Theorem \ref{auxiliary_theorem}: average weight of isolated piles}\label{sec5}

Throughout this section, we let $n,d,c\in\mathbb{N}$ with $d\ge 50$ and $c\in [d]$, $\g \subseteq 2^{[n]}$ be a hereditary family with $\md(\g)\geq 2^{d-1}-c+1$, and $P\subseteq [n]$ be an isolated pile in $\g$. 
We will assume without loss of generality that $P=[d]$.
Our goal is to demonstrate that $\sum_{x\in [d]}\ww_\g(x)\geq \wb d$.
Before proceeding with the proof, it is necessary to establish several definitions.

Recall that we assume $P=[d]$. 
Define $\f=\g|_{[d]}=\{\set\subseteq[d]\mid\set\in\g\}$ be the projection of the family $\g$ onto the pile $P$.
Then $\f$ is also a hereditary family.
For a set $\set\in \g$, we say it is {\it internal} if $\set\in\f$, and {\it external} if $\set\in\g\setminus\f$.
Let $x\in [d]$. We define $\win(x)=\sum_{x\in \set\in \f}\frac{1}{|\set|}$ be the internal weight of $x$, $\wout(x)=\sum_{x\in \set\in \g\setminus\f}\frac{1}{|\set|}$ be the external weight of $x$.
Clearly for any $x\in [d]$, $\ww(x)=\win(x)+\wout(x)$.
For any $x\in [d]$, we denote the number of sets in $\f$ containing $x$ by $d_\f(x)$, 
while the number of sets in $\g$ containing $x$ is denoted by $d_\g(x)$. 
Let $$\md=2^{d-1}-c+1$$ be a constant. 
So for any $x\in [d]$, we have $d_\g(x)\geq \md(\g)\geq \md$.

We assume $|\f|=2^d-t$ and let $\mf=2^{[d]}\setminus\f=\{\m_1,\m_2,\dots,\m_t\}$ be the family consisting of all missing sets of $\f$. 
Then $\f=2^{[d]}\setminus\mf=2^{[d]}\setminus\{\m_1,\m_2,\dots,\m_t\}$.
To enhance the readability of the proof, we also define $\n_i=[d]\setminus \m_i$ and let $\nf=\{\n_1,\n_2,\dots,\n_t\}$.
By definition, it follows that $\nf$ is a hereditary family if and only if $\f$ is a hereditary family. 
So both families $\f$ and $\nf$ are hereditary families. Lastly, we note that for any $x\in [d]$, 
\begin{equation}\label{equ:d_N}
    d_\f(x)=d_{2^{[d]}}(x)-d_\mf(x)=2^{d-1}-d_\mf(x)=2^{d-1}-t+d_\nf(x).
\end{equation}

The following lemma collects essential technical information required for proving Theorem \ref{auxiliary_theorem}.

\begin{lemma}\label{few_good_vertex_theorem}
Let $n,d,c\in\mathbb{N}$ with $d\ge 50$ and $c\in [d]$, $\g \subseteq 2^{[n]}$ be a hereditary family with $\md(\g)\geq 2^{d-1}-c+1$, and $P\subseteq [n]$ be an isolated pile in $\g$. Write $P=[d]$. If $$\sum_{x\in [d]}\ww_\g(x)< 2^d-c,$$
then the following four statements hold:
\begin{itemize}
    \item there are at most $7$ good vertices in $[d]$,
    \item $t\leq d+4$,
    \item for any $i\in [t]$, $|\n_i|\leq 3$, and
    \item for every bad vertex $x$ in $[d]$, $\{x\}\in \nf$.
\end{itemize}
\end{lemma}

We will postpone the proof of this lemma in Subsection~\ref{subsec:iso-pile}.
Now we prove Theorem \ref{auxiliary_theorem}.

\begin{proof}[\bf{Proof of Theorem \ref{auxiliary_theorem} (assuming Lemma \ref{few_good_vertex_theorem})}]
Suppose for a contradiction that
there exists an isolated pile $P=[d]$ of $\mathcal{F}$
with $\sum_{x\in [d]}\ww_\g(x)< \wb d$. 
Then we have 
\[
\sum_{x\in [d]}\ww_\g(x)< \wb d\leq 2^d-c.
\]
By Lemma~\ref{few_good_vertex_theorem}, 
there are at most $7$ good vertices in $[d]$, $t\leq d+4$, $|\n_i|\leq 3$ for any $i\in [t]$, and for every bad vertex $x$ in $[d]$, $\{x\}\in \nf$. 
We may assume that all vertices in $[d-7]$ are bad vertices. 
Therefore we have $\{\emptyset\}\cup\{\{x\}\mid x\in [d-7]\}\subseteq \nf$.
Then the family $\nf$ has at most $|\nf|-(d-7)-1=t-d+6\leq 10$ sets of size greater than $1$. 
As $|\n_i|\leq 3$ for any $i\in [t]$, those sets can cover at most $30$ vertices. 
Thus there are at least $d-7-30\geq 13$ bad vertices that are contained in exactly one set in $\nf$, which is the singleton set formed by itself. 
Assume $x\in [d-7]$ is one of these bad vertices.
Then $2^{d-1}-c+1=\md\leq d_\g(x)=2^{d-1}-(t-1)$. Therefore $t\leq c$.

First consider when $1\leq c\leq d-1$.
We have $\emptyset\in \nf$ and for every bad vertex $x$, $\{x\}\in \nf$.
Since $|\nf|=t\leq c\leq d-1$, there are at most $d-2$ bad vertices and thus at least $2$ good vertices in $[d]$.
For a good vertex $x\in [d]$, it has at least one neighbor $y\notin [d]$, so $\{x,y\}\in \g\setminus\f$, implying that 
\begin{equation}\label{equ:wout>1/2}
\wout(x)=\sum_{x\in \set\in \g\setminus\f}\frac{1}{|\set|}\geq \frac12.
\end{equation}
It is easy to see $\sum_{x\in [d]}\win(x)=|\f\setminus\{\emptyset\}|=2^d-t-1$.
Thus, we have $$\sum_{x\in [d]}\ww(x)=\sum_{x\in [d]}\win(x)+\sum_{x\in [d]}\wout(x)\geq (2^d-t-1)+2\cdot \frac{1}{2}\geq2^d-c=\wb d,$$
a contradiction to our assumption.

Now we consider when $c=d$. 
Again we have $\emptyset\in \nf$ and for every bad vertex $x$, $\{x\}\in \nf$.
Since $|\nf|=t\leq c=d$, there are at most $d-1$ bad vertices and at least $1$ good vertex in $[d]$. Thus $$\sum_{x\in [d]}\ww(x)=\sum_{x\in [d]}\win(x)+\sum_{x\in [d]}\wout(x)\geq|\f\setminus\{\emptyset\}|+\frac{1}{2}\geq (2^d-t-1)+\frac{1}{2}\geq2^d-d-\frac{1}{2}=\wb d,$$
a contradiction to our assumption.
Thus we finish the proof of Theorem \ref{auxiliary_theorem}.
\end{proof}

We add a remark that in the case of $c=d$, the equality of Theorem \ref{combined_theorem} holds if and only if the hereditary family $\mathcal{F}$ is isomorphic to $\mathcal{F}_{0}(n,d)$ given in Construction \ref{construction_c=d}. 
To see this, suppose that we have the equation $\sum_{x\in [n]}\ww_\g(x)=\wb n$ in the proof of Theorem \ref{combined_theorem}.
Note that both inequalities \eqref{equ:J} and \eqref{equ:K} are strict.
So this forces that there is no vertices in $J\cup K$, i.e.,
the vertex set of $\mathcal{F}$ consists of isolated piles $P$; 
moreover, for each isolated pile $P$ we have $\sum_{x\in P}\ww_\g(x)=\wb d$ in Theorem~\ref{auxiliary_theorem}. 
In the above proof of Theorem~\ref{auxiliary_theorem}, we see that this equality holds if and only if there are $d-1$ bad vertices and one good vertex with exactly one external edge of size $2$,
where $\nf$ consists of the empty set and $d-1$ singleton sets formed by those $d-1$ bad vertices.
By combining all the aforementioned information, it becomes clear that $\mathcal{F}$ is isomorphic to $\mathcal{F}_{0}(n,d)$ in Construction \ref{construction_c=d}.

On the other hand, in the case $1\leq c\leq d-1$, there are situations where multiple hereditary families exist for which the equality in Theorem \ref{combined_theorem} holds.

\subsection{Proof of Lemma~\ref{few_good_vertex_theorem}}\label{subsec:iso-pile} 
Let $P=[d]$ be an isolated pile in a hereditary family $\g\subseteq 2^{[n]}$ with $\delta(\mathcal{F})\geq 2^{d-1}-c+1$, where $d\geq 50$ and $c\in [d]$. 
Under the assumption that $\sum_{x\in [d]}\ww_\g(x)< 2^d-c$, we want to verify the four conclusions of Lemma~\ref{few_good_vertex_theorem}.
First we establish several simple facts.

\begin{lemma}\label{facts}
Let $P=[d]$ be an isolated pile. Then the following hold:

    (a) $t\geq c$.

    (b) $t\leq 2c-2$.

    (c) For any $ x\in [d]$, $d_\f(x)\geq 2^{d-1}-2c+2$.

    (d) If $x$ is a bad vertex, then $\{x\}\in \nf$.

    (e) There are at most $\frac{d}{2}-1$ good vertices in $[d]$.
\end{lemma}
\begin{proof}

(a) Suppose for a contradiction that $t\leq c-1$. 
By counting the weights in $[d]$, we get $$\sum_{x\in[d]}\ww(x)\geq\sum_{x\in[d]}\win(x)=|\f\setminus \{\emptyset\}|=2^d-t-1\geq 2^d-c$$ which is a contradiction. So we have $t\geq c$.

(b) By the definition of piles, there exists at least one bad vertex say $x$ in $[d]$. 
Then $|\{\set \in \f\mid x \in \set\}|=d_\f(x)=d_\g(x)\geq \md= 2^{d-1}-c+1$.
Since $\f$ is a hereditary family, it holds that 
\[
|\{\set \in \f\mid x \not\in \set\}|\geq|\{\set \in \f\mid x \in \set\}|\geq 2^{d-1}-c+1.
\]
Consequently, $2^d-t=|\f|=|\{\set \in \f\mid x \not\in \set\}|+|\{\set \in \f\mid x \in \set\}|\geq 2^{d}-2c+2$. Thus $t\leq 2c-2$.

(c) Recall that $\f=2^{[d]}\setminus \mf$. 
Using \eqref{equ:d_N} and the above (b), we have $d_{\f}(x)\geq d_{2^{[d]}}(x)-|\mf|=2^{d-1}-t\geq 2^{d-1}-2c+2.$

(d) If $x$ is contained by every set in $\mf$, then $d_\g(x)=d_\f(x)=d_{2^{[d]}}(x)-|\mf|=2^{d-1}-t\leq 2^{d-1}-c<\md$, a contradiction. 
So there exists one set in $\mf$ not containing $x$. 
That is, there exists one set in $\nf$ containing $x$. 
Since $\nf$ is a hereditary family, we have $\{x\}\in \nf$.

(e) Suppose not. Then using Lemma \ref{weight_lemma} and $d\geq 50$, $\sum_{x\in [d]}\ww(x)> (\wb-\frac{1}{18})\cdot d +\frac{1}{6}\cdot (\frac{d}{2}-1)> d\wb+\frac{1}{2}\geq 2^d-c$, which is a contradiction.
\end{proof}

Before proceeding with the proof, we prove a lemma that allows us to control the external weight.

\begin{lemma}\label{externalweight}
For any vertex $x\in [d]$, it holds that $\wout(x)\geq \frac{\md-d_\f(x)}{3+\ln c}$.
\end{lemma}
\begin{proof}
If $\md\leq d_\f(x)$ this is trivial. So we may assume $\md> d_\f(x)$. 
This implies that $x$ is a good vertex and there are at least $\md-d_\f(x)$ many sets $\set$ satisfying that $x\in \set\in \g\setminus\f$.
If all these sets have size at most $\ln c+3$, then it follows that $$\wout(x)=\sum_{x\in \set\in \g\setminus\f}\frac{1}{|\set|}\geq \frac{d_\g(x)-d_\f(x)}{\ln c+3}\geq \frac{\md-d_\f(x)}{\ln c+3},$$
as desired.
Therefore we may assume there is a set $ Q$ such that $x\in  Q\in \g\setminus\f$, $| Q|> \ln c+3$ and there exists $y \in  Q \setminus [d]$. 
We can choose a set $Q^\prime$ with $ Q^\prime \subseteq  Q\setminus\{x,y\}$ and $| Q^\prime|=\lceil\ln c\rceil$. 
Define $\h=\{\{x,y\}\cup\set\mid\set\subseteq Q^\prime\}$. 
Since $\g$ is hereditary, every set in $\h$ is contained in $Q$ and contains the vertex $y$, 
implying that $\h\subseteq\g\setminus\f$.
Therefore, we can derive
$$\wout(x)\geq\sum_{\set\in \h}\frac{1}{|\set|}\geq \frac{2^{\lceil\ln c\rceil}}{\lceil\ln c\rceil+2}\geq \frac{(2^{d-1}-c+1)-(2^{d-1}-2c+2)}{\ln c+3}\geq \frac{\md-d_\f(x)}{\ln c+3},$$
finishing the proof.
\end{proof}

Next, we define a sequence of integers to assist in upper bounding the size of $N_i$.

\begin{definition}
For any positive integer $u$, let $f_u$ be the unique integer such that $$2^{f_u}-f_u\leq u < 2^{f_u+1}-(f_u+1).$$
Since $\max\{f_u, 2^{f_u}/2\}\leq 2^{f_u}-f_u$, 
it is easy to derive that $f_u\leq u$, $f_u\leq \log (2u)$ and $u \leq 2^{f_u+1}-f_u-2$.
\end{definition}
\begin{lemma}\label{size_lemma}
    Let $u$ be the number of good vertices in $[d]$. 
    Then for any $i\in [t]$, we have $|\n_i|\leq f_u$.
\end{lemma}
\begin{proof}
By Lemma \ref{facts} (e), we have $u\leq \frac{d}{2}-1$. 
Suppose for a contradiction that $|\n_i|\geq f_u+1$. 
Since $\nf$ is hereditary, there exists a set $\n\in\nf$ such that $|\n|= f_u+1$ and $2^\n\subseteq\nf$.
Let $\m=[d]\setminus\n$, then $|\m|=d-f_u-1\geq d-u-1\geq \frac{d}{2}\geq u+1$.
There are $u$ good vertices in total, 
so there exists $Q\subseteq \m$ with $| Q|=u+1$ such that $ Q$ contains all the good vertices in $\m$.
Then all the vertices in $\m\setminus Q$ are bad vertices. 
By Lemma \ref{facts} (d), for any $x\in\m\setminus Q$, $\{x\}\in \nf$.
Therefore, $\nf^\prime:=2^\n\cup\{\{x\}\mid x\in \m\setminus Q \}$ is a subfamily of $\nf$ with size $$2^{f_u+1}+(d-f_u-1)-(u+1)\geq2^{f_u+1}+(d-f_u-1)-(2^{f_u+1}-f_u-2+1)\geq d.$$
Note that $\nf^\prime$ does not contain any vertex $y\in Q$. 
Hence there are at least $d$ missing sets in $\mf $ containing $y$, 
implying that $d_\f(y)\leq 2^{d-1}-d<\md$. 
Thus every $y\in Q$ must be a good vertex. 
As a result, there are at least $|Q|=u+1$ good vertices, a contradiction to our definition.
\end{proof}

The following lemma is the most technical part in our proof of  Lemma~\ref{few_good_vertex_theorem}.

\begin{lemma}\label{dijiang}
   Let $u$ be the number of good vertices in $[d]$. Then $u\leq 7$.
\end{lemma}
\begin{proof}
Suppose for a contradiction that $u\geq 8$.
By Lemma \ref{size_lemma}, $|\n_i|\leq f_u$ for each $i\in [t]$,
which implies that $|\m_i|\geq d-f_u$. 
Straightforward calculations yield the following bounds:
$$\sum_{x\in [d]}d_\f(x)=\sum_{\set\in\f}|\set|=d2^{d-1}-\sum_{i=1}^t|\m_i|\leq d2^{d-1}-t(d-f_u), \mbox{ and}$$
\begin{equation}\label{equ:delta-dx}
\sum_{x\in [d]}(\md-d_\f(x))\geq d(2^{d-1}-c+1)-(d2^{d-1}-t(d-f_u))=(d-f_u)t+d-cd.
\end{equation}
Since there are $u$ good vertices in $[d]$, using \eqref{equ:wout>1/2} we have $\sum_{x\in[d]}\wout(x)\geq \frac{u}{2}$, which implies that
 $$0\geq\sum_{x\in[d]}\ww(x)-(2^d-c)\geq\sum_{x\in[d]}\win(x)+\sum_{x\in[d]}\wout(x)-(2^d-c)\geq (2^d-t-1)+\frac{u}{2}-(2^d-c)\geq c-t-1+\frac{u}{2}.$$
So $t\geq c-1+\frac{u}{2}$. 
Note that $u\leq \frac{d}{2}-1$, thus $f_u\leq \ln (2u) \leq \ln d$. 
Using Lemma \ref{externalweight} and \eqref{equ:delta-dx},
$$
0\geq\sum_{x\in[d]}\ww(x)-(2^d-c)\geq\sum_{x\in[d]}\win(x)+\sum_{x\in[d]}\wout(x)-(2^d-c)
\geq (2^d-t-1)+\frac{(d-f_u)t+d-cd}{\ln c+3}-(2^d-c).
$$
The partial derivative of the right-hand side with respect to $t$ is $-1+\frac{d-f_u}{\ln c+3}\geq-1+\frac{d-\ln d}{\ln d+3}\geq 0$. 
Thus using the lower bound $t\geq c-1+\frac{u}{2}$, we can further obtain that 
$$
0\geq \left(2^d-\left(c-1+\frac{u}{2}\right)-1\right)+\frac{(d-f_u)(c-1+\frac{u}{2})+d-cd}{\ln c+3}-(2^d-c)= -\frac{u}{2}+\frac{(d-f_u)\frac{u}{2}-f_u(c-1)}{\ln c+3}.
$$
By simplifying this expression and using $c\leq d$, we obtain that
\begin{equation}\label{equ:u-1}
0\geq (d-f_u-\ln c-3) -\frac{2f_u}{u}(c-1)\geq (d-f_u-\ln d-3) -\frac{2f_u}{u}(d-1).
\end{equation}

If $f_u=3$ or $4$, then the partial derivative of the right-hand side of \eqref{equ:u-1}
with respect to $d$ is $$1-\frac{1}{d\oldln 2}-\frac{2f_u}{u}\geq\min\left\{1-\frac{1}{d\oldln 2}-\frac{2\times 3}{8},1-\frac{1}{d\oldln 2}-\frac{2\times 4}{12}\right\}\geq 0,$$
where the second inequality uses the assumption $u\geq 8$ and the fact $u\geq 12$ when $f_u=4$, 
thus the right-hand side of \eqref{equ:u-1} is increasing as $d$ grows. 
Using $d\geq 50$, we arrive the following contradiction
$$
\begin{aligned}
    0&\geq (50-f_u-\ln 50-3) -\frac{98f_u}{u}\geq 49\left(1-\frac{2f_u}{u}\right)- (f_u+8)\\
&\geq\min\left\{49\left(1-\frac{2\times 3}{8}\right)- (3+8),49\left(1-\frac{2\times 4}{12}\right)-(4+8)\right\}>0.
\end{aligned}
$$
Now assume $f_u\geq 5$. Using $f_u\leq \ln d$ and $\frac{f_u}u\leq \frac{f_u}{2^{f_u}-f_u}\leq \frac{5}{27}$, the inequality \eqref{equ:u-1} implies that
 $$0\geq(d-2\ln d-3)-\frac{10}{27}(d-1)\geq \frac{3}{5}d-2\ln d-3>0,$$
where the last inequality holds when $d\geq 50$.  
This contradiction completes Lemma~\ref{dijiang}.
\end{proof}

Finally, we are prepared to prove Lemma~\ref{few_good_vertex_theorem}.

\begin{proof}[\bf Proof of Lemma~\ref{few_good_vertex_theorem}]
By Lemma~\ref{dijiang}, there are $u\leq 7$ good vertices in $[d]$. By Lemma \ref{size_lemma}, for any $i\in [t]$, we have $|\n_i|\leq f_u\leq f_7=3$.
Then using Lemma \ref{externalweight} and \eqref{equ:delta-dx}, we obtain 
$$
0\geq\sum_{x\in[d]}\ww(x)-(2^d-c)=\sum_{x\in[d]}\win(x)+\sum_{x\in[d]}\wout(x)-(2^d-c)
\geq (2^d-t-1)+\frac{(d-3)t+d-cd}{\ln c+3}-(2^d-c).
$$
Then using $c\leq d, d\geq 50$ and $\frac{d}{d-\log d-6}\leq \frac{5}{3}$, it can be deduced that
$$
t\leq \frac{(c-1)(d-\ln c-3)}{d-\ln c-6}=c-1+\frac{3(c-1)}{d-\ln c-6}\leq d-1+\frac{3d}{d-\ln d-6}\leq d+4.
$$
Finally, by Lemma \ref{facts} (d), for any bad vertex $x$ in $[d]$, we have $\{x\}\in \nf$.
\end{proof}

\section{Concluding remarks}

In this paper, we investigate extremal problems regarding to the arrowing relation $(n,m)\to (a,b)$ and primarily focus on determining the limiting constant $m(2^{d-1}-c)$ for all $1\leq c\leq d$.
Through the use of novel concepts and analysis,
we prove a conjecture posed by Frankl and Watanabe \cite{WF1994} by showing $m(11)=5.3$. Furthermore, we provide the exact value of $m(2^{d-1}-c)$ for all $1\leq c\leq d$ when $d\geq 50$, thereby essentially resolving an open problem raised by Piga and Sch\"{u}lke \cite{PS2021}.

We would like to discuss some natural questions for future study.
It would be interesting to reduce the bound $d\geq 50$ to $d\geq 6$ in Theorem~\ref{zhudingli}.
It is worth noting that by refining the condition in the proof of Lemma~\ref{appoximation_lemma} from $c\le 2^{d}$ to $c\le d$, it is possible to improve the leading coefficient from $\frac{1}{6}$ to $\frac{1}{2}-\frac{1}{d-\ln e}$. 
This refinement could potentially enhance the bound in Theorem~\ref{zhudingli} from $d\geq 50$ to around $d\geq 28$. 
However, it is evident that closing the gap between $d=5$ and $d\geq 50$ may require the introduction of some innovative ideas.

The ultimate question in this direction would be to determine $m(s)$ for all $s\in \mathbb{N}$, which is equivalent to the question of
\begin{equation}\label{gen-ques}
\mbox{ determining  $m(2^{d-1}-c)$ for all $1\leq c\leq 2^{d-2}$.}
\end{equation}
Currently, it seems beyond our reach to provide a complete solution.
However, we will now discuss a potential recursive method for determining $m(2^{d-1}-c)$ for an infinite sequence of $c\geq d+1$.

\begin{definition}
    Let $n,s\in \mathbb{N}$. We call $\mathcal{F}$ an {\it $(n,s)$-extremal family},\footnote{This definition neither guarantees uniqueness nor existence. For a specified pair $(n,s)$, there might be many distinct $(n,s)$-extremal families, or there might not be any at all.} if
\begin{enumerate}
    \item $\mathcal{F}\subseteq 2^{[n]}$ is a hereditary family,
    \item $|\mathcal{F}|=m(s)\cdot n+1$, and
    \item for any $x\in [n]$, $d_{\mathcal{F}}(x)\geq s+1$.
\end{enumerate}
\end{definition}

Using this notion, we see that each $d$-set $U_i$ in Construction \ref{construction_c<d} provides an $(d,2^{d-1}-c)$-extremal family for $1\leq c\leq d-1$,
and each $2d$-set $U_{2i-1}\cup U_{2i}$ in Construction \ref{construction_c=d} provides an $(2d,2^{d-1}-d)$-extremal family. 
We observe the following fact.
\begin{fact}
Assume $\mathcal{F}$ is an $(n,s)$-extremal family. 
Write $F^c=[n]\setminus F$ for all $F\subseteq [n]$.
Let $\mathcal{F}^*=2^{[n]}\setminus\{F^c \mid F\in \mathcal{F}\}$. 
Then, the following hold:
\begin{itemize}
    \item $\mathcal{F}^*\subseteq 2^{[n]}$ is a hereditary family,
    \item $|\mathcal{F}^*|=2^n-|\mathcal{F}|=2^n-(m(s)\cdot n+1)$, and
    \item For any $x\in [n]$, $d_{\mathcal{F}^*}(x)=2^{n-1}-|\mathcal{F}|+d_{\mathcal{F}}(x)\geq 2^{n-1}-m(s)\cdot n+s$.
\end{itemize}
\end{fact}

Based on this fact, we propose the following question.
\begin{question}\label{ques}
Let $n,s\in \mathbb{N}$ satisfy that $m(s)\cdot n+1\leq 2^{n-3}$ and $s\leq 2^{\frac{n}{2}-1}-1$. 
If $\mathcal{F}$ is an $(n,s)$-extremal family, 
determine whether $\mathcal{F}^*$ is an $(n,2^{n-1}+s-m(s)\cdot n-1)$-extremal family.
\end{question}

If this holds true, it would imply that for any $n$ and $s$ satisfying the conditions of Question \ref{ques}, it holds that $$m(2^{n-1}+s-m(s)\cdot n-1)=\frac{2^n-m(s)\cdot n-2}{n}.$$
In Appendix~\ref{sec:F^star}, we provide two examples that assume the validity of Question~\ref{ques}. 
These examples aim to shed some light on the general question of \eqref{gen-ques}.

Finally, we would like to direct interested readers to the paper by Piga and Sch\"{u}lke \cite{PS2021}, where they discuss and formalize several intriguing open problems related to the arrowing relation.   

\bibliographystyle{unsrt}

\appendix
\section{Proof of an inequality in Lemma \ref{weight_lemma}}\label{compute}
In this section, we will perform detailed calculations to establish the validity of the following inequality for $d\geqslant 50$:
\[
h(d):=\sum_{H\in \mathcal{R}(d-1)}\left(\frac{1}{d-|H|}-\frac{1}{d}\right)<\frac{1}{18}. 
\]

{\noindent \it Proof.}
Since for any $H\in \mathcal{R}(d-1)$, we have $|H|\leqslant \log(d-1)$. 
So it is easy to see that 
\[
h(d)\leqslant \sum_{H\in \mathcal{R}(d-1)}\left(\frac{1}{d-\log(d-1)}-\frac{1}{d}\right)=\frac{(d-1)\log(d-1)}{d(d-\log(d-1))}:=h_{1}(d). 
\]
When $d\geqslant 133$, we can derive that 
\[
h_{1}'(d)=\frac{-\log^{2}(d-1)-(d^{2}-2d)\log(d-1)+d^{2}}{d^{2}(d-\log(d-1))^{2}}<\frac{-\log^{2}(d-1)-\frac{d^{2}}{2}\log(d-1)+d^{2}}{d^{2}(d-\log(d-1))^{2}}<0
\]
Since $h_{1}(133)<\frac{1}{18}$, we can obtain that when $d\geqslant 133$, 
\[
h(d)\leqslant h_1(d)\leq h_{1}(133)<\frac{1}{18}. 
\]
When $50\leqslant d\leqslant 64$, we have $\mathcal{R}(49)\subseteq \mathcal{R}(d-1)\subsetneq 2^{[6]}$, where 
\[\mathcal{R}(49)=2^{[5]}\cup \{F\cup \{6\}:F\subseteq [4]\}\cup \{\{5,6\}\}.\]
Thus in this case, we have 
\begin{align}
    h(d)&\leqslant \sum_{H\in \mathcal{R}(49)}\frac{1}{d-|H|}+\frac{d-50}{d-5}-\frac{d-1}{d}=h_2(d),  \notag 
\end{align}
where 
\[h_{2}(d)=\frac{2}{d}+\frac{6}{d-1}+\frac{15}{d-2}+\frac{16}{d-3}+\frac{9}{d-4}-\frac{43}{d-5}.\]
Using Jensen's inequality, we have 
\begin{align}
    h_{2}'(d)&=\frac{43}{(d-5)^{2}}-\left(\frac{2}{d^{2}}+\frac{6}{(d-1)^{2}}+\frac{15}{(d-2)^{2}}+\frac{16}{(d-3)^{2}}+\frac{9}{(d-4)^{2}}\right) \leqslant \frac{43}{(d-5)^{2}}-\frac{48}{(d-\frac{5}{2})^{2}}<0. \notag  
\end{align}
Since $h_{2}(50)<\frac{1}{18}$, we can derive that when $50\leqslant d\leqslant 64$, it holds that 
\[
h(d)\leq h_2(d)\leqslant h_{2}(50)<\frac{1}{18}. 
\]
Similarly, when $65\leqslant d\leqslant 128$, since $2^{[6]}\subseteq \mathcal{R}(d-1)\subsetneq 2^{[7]}$, we have 
\begin{align}
h(d)&\leqslant \sum_{H\subseteq [6]}\frac{1}{d-|H|}+\frac{d-65}{d-6}-\frac{d-1}{d}= h_3(d),  \notag 
\end{align}
where  
\[
h_{3}(d)=\frac{2}{d}+\frac{6}{d-1}+\frac{15}{d-2}+\frac{20}{d-3}+\frac{15}{d-4}+\frac{6}{d-5}-\frac{58}{d-6}.
\]
Using Jensen's inequality, we can derive the following 
\begin{align}
    h_{3}'(d)&=\frac{58}{(d-6)^{2}}-\left(\frac{2}{d^{2}}+\frac{6}{(d-1)^{2}}+\frac{15}{(d-2)^{2}}+\frac{20}{(d-3)^{2}}+\frac{15}{(d-4)^{2}}+\frac{6}{(d-5)^{2}}\right) \notag \\
    &\leqslant \frac{58}{(d-6)^{2}}-\frac{64}{(d-\frac{93}{32})^{2}}<0. \notag  
\end{align}
Therefore, since $h_{3}(68)<\frac{1}{18}$, we have that when $68\leqslant d\leqslant 128$, 
\[
h(d)\leq h_3(d)\leqslant h_{3}(68)<\frac{1}{18}. 
\]
To complete the proof, we will now consider the remaining cases when $65\leqslant d\leqslant 67$ and $129\leqslant d\leqslant 132$. 
We will calculate the formula for each case and provide a list of them as follows:
\begin{align}
&h(65)=\frac{1}{65}+\frac{6}{64}+\frac{15}{63}+\frac{20}{62}+\frac{15}{61}+\frac{6}{60}+\frac{1}{59}-\frac{64}{65}\approx 0.048<\frac{1}{18} \notag \\
&h(66)=\frac{1}{66}+\frac{7}{65}+\frac{15}{64}+\frac{20}{63}+\frac{15}{62}+\frac{6}{61}+\frac{1}{60}-\frac{65}{66}\approx 0.047<\frac{1}{18} \notag \\
&h(67)=\frac{1}{67}+\frac{7}{66}+\frac{16}{65}+\frac{20}{64}+\frac{15}{63}+\frac{6}{62}+\frac{1}{61}-\frac{66}{67}\approx 0.046<\frac{1}{18} \notag \\
&h(129)=\frac{1}{129}+\frac{7}{128}+\frac{21}{127}+\frac{35}{126}+\frac{35}{125}+\frac{21}{124}+\frac{7}{123}+\frac{1}{122}-\frac{128}{129}\approx 0.028<\frac{1}{18} \notag \\
&h(130)=\frac{1}{130}+\frac{8}{129}+\frac{21}{128}+\frac{35}{127}+\frac{35}{126}+\frac{21}{125}+\frac{7}{124}+\frac{1}{123}-\frac{129}{130}\approx 0.027<\frac{1}{18} \notag \\
&h(131)=\frac{1}{131}+\frac{8}{130}+\frac{22}{129}+\frac{35}{128}+\frac{35}{127}+\frac{21}{126}+\frac{7}{125}+\frac{1}{124}-\frac{130}{131}\approx 0.027<\frac{1}{18} \notag \\
&h(132)=\frac{1}{132}+\frac{8}{131}+\frac{23}{130}+\frac{35}{129}+\frac{35}{128}+\frac{21}{127}+\frac{7}{126}+\frac{1}{125}-\frac{131}{132}\approx 0.027<\frac{1}{18}. \notag 
\end{align}
In summary, we have demonstrated that for any $d\geqslant 50$, the inequality
\[
h(d)=\sum_{H\in \mathcal{R}(d-1)}\left(\frac{1}{d-|H|}-\frac{1}{d}\right)<\frac{1}{18}
\]
holds, as desired. \qed

\section{Two examples assuming the validity of Question~\ref{ques}}\label{sec:F^star}

Assuming Question~\ref{ques} is valid, we would be able to determine $m(2^{d-1}-c)$ for many values of $c\geq d+1$. 
Here, we present two explicit examples based on $(n,s)$-extremal hereditary families, specifically focusing on the cases where $s=0$ and $s=1$.
Assuming the validity of Question~\ref{ques} in these cases, one can determine infinitely many values of $m(2^{d-1}-c)$ for $c=d+1$ and $c=\frac{3}{2}d$.

\begin{example}[For $s=0$] For $d\geq 6$, the following family
$$\mathcal{F}=\{\emptyset\}\cup\{\{x\}\mid x\in [d]\}$$ is a $(d,0)$-extremal family. 
Assume that Question~\ref{ques} holds in this case. Then
$$\mathcal{F}^*=2^{[d]}\setminus\{[d],\{1\}^c,\{2\}^c,\dots,\{d\}^c\}$$ 
becomes a $(d,2^{d-1}-d-1)$-extremal family, which implies $m(2^{d-1}-d-1)=\frac{2^d-d-2}{d}$ for any $d\geq 6$.
\end{example}

\begin{example}[For $s=1$] Let $d\geq 8$ and $2\mid d$. The following family
$$\mathcal{F}=\{\emptyset\}\cup\{\{x\}\mid x\in [d]\}\cup \{\{x,x+1\}\mid x\in [d] \text{ is odd}\}$$ is a $(d,1)$-extremal family.
Assume that Question~\ref{ques} holds in this case. Then
$$\mathcal{F}^*=2^{[d]}\setminus\{[d],\{1\}^c,\{2\}^c,\dots,\{d\}^c,\{1,2\}^c,\{3,4\}^c,\dots,\{d-1,d\}^c\}$$is a $(d,2^{d-1}-\frac{3}{2}d)$-extremal family,
which implies $m(2^{d-1}-\frac{3}{2}d)=\frac{2^d-\frac{3}{2}d-2}{d}$ for any even $d\geq 8$.
\end{example}

\end{document}